\documentclass{amsart}
\usepackage{amssymb}
\usepackage{soul}
\usepackage{xcolor}
\usepackage{tikz}
\usetikzlibrary{arrows}
\usepackage{hyperref}
\usepackage{float}

\newcommand\Q{{\mathbb Q}}
\newcommand\R{{\mathbb R}}
\newcommand\C{{\mathbb C}}
\newcommand\Z{{\mathbb Z}}
\newcommand\N{{\mathbb N}}

\usepackage{setspace} 
\renewcommand\S{{\mathbb S}}
\renewcommand\P{{\mathbb P}}

\newtheorem{theorem}{Theorem}[section]
\newtheorem*{theorem*}{Theorem}

\newtheorem{lemma}[theorem]{Lemma}
\newtheorem*{claim}{Claim}
\theoremstyle{definition}
\newtheorem{example}[theorem]{Example}
\newtheorem{remark}[theorem]{Remark}
\newtheorem*{remark*}{Remark}
\newtheorem{definition}[theorem]{Definition}
\newcommand{\lcm}{\operatorname{lcm}}
\DeclareMathOperator{\len}{length}
\DeclareMathOperator{\dist}{dist}
\DeclareMathOperator{\ord}{ord}

\title{Lipschitz geometry of complex algebraic plane curves}
\author{Renato Targino}
\address{Aix Marseille Univ, CNRS, Centrale Marseille, I2M, Marseille, France}
\address{Departamento de Matem\'atica,Universidade 
Federal do Cear\'a,
Fortaleza-CE, Brazil}
\address{Instituto Federal de Educa\c c\~ao, Ci\^encia e Tecnologia do Cear\'a, Maracana\'u-CE, Brazil}

   \email{renatotargino@ifce.edu.br}

\begin{document}

\maketitle
\section{Introduction}

\begin{abstract}
We present a complete classification of complex plane algebraic curves, equipped with the induced Euclidean, up to global bilipschitz homeomorphism.      
\end{abstract}

One of the most natural questions in the investigation of a  class of mathematical objects is  the problem of classification of these objects. Here the classification problem is treated from the outer metric viewpoint: all  the  subsets  of the euclidean space $\R^n$ are  considered  equipped  with  the induced Euclidean metric. 
Our objects are complex algebraic plane curves and we obtain a complete classification of them up to global bilipschitz homeomorphism.
In order to present precisely our results let us introduce some definitions and notations.

\begin{definition}
Let $(M,d)$ and $(M',d')$ be two metric spaces. 
A map $f:M\to M'$ is \textbf{Lipschitz} if  there exists a real constant $c>0$ such that 
$$ d'(f(x),f(y))\leq cd(x,y)\text{ for all }x,y\in M.$$ 

A Lipschitz map $f:M\to M'$ is called \textbf{bilipschitz} if its inverse  exists and it is Lipschitz.  We say that $M$ and $M'$ are \textbf{bilipschitz equivalent} if there exists a bilipschitz map $f:M\to M'$ between them.  
The equivalence class of $M$ in this relation is called the \textbf{Lipschitz geometry} of $M$.
\end{definition}

One  of  the  recent  works  on Lipschitz geometry, Neumann and Pichon \cite{Anne} proved that two germs of plane complex curves are bilipschitz homeomorphic if only if they have the same topological type, the meaning of topological type here is in the following definition.

\begin{definition}
Let $(C_1,p_1)\subset (S_1,p_1)$ and   $(C_2,p_2)\subset (S_2,p_2)$ be two germs of complex curves on smooth surfaces. 
We say that $(C_1,p_1)$ and $(C_2,p_2)$ have the same \textbf{topology type} if there is a homeomorphism of germs $h \colon (S_1,p_1) \to (S_2,p_2)$ such that $h(C_1)=C_2$. 
\end{definition}

Previous contributions on the problem of classification of germs of plane complex curves up to bilipschitz equivalence were made by  Fernandes \cite{fernandes2003}, and Pham and Teissier \cite{pham:hal-00384928}. 
Let us point out that the theorems of Pham and Teissier \cite{pham:hal-00384928}, Fernandes \cite{fernandes2003}, and Neumann and Pichon \cite{Anne} are local results.    

 On the other hand, looking to scrutinize global Lipschitz geometry of algebraic sets, Fernandes and Sampaio \cite{Fernandes2019} arrived on the notion of bilipschitz equivalence at infinity of subsets in  Euclidean space.
 
 \begin{definition}\label{bi-Lipschitz equivalent at infinity}
Let $X\subset \R^n$ and $Y\subset\R^m$ be two subsets.
We say that $X$ and $Y$ are {\bf bilipschitz equivalent at infinity} if there exist compact subsets $K\subset\R^n$ and $\widetilde K\subset \R^m$, and a bilipschitz map $\Phi \colon X\backslash K\rightarrow Y\backslash \widetilde K$. 
The equivalence class of $X$ in this relation is called the \textbf{Lipschitz geometry  at infinity} of $X$.
\end{definition}
 
 One of the goals of this paper is to bring a complete bilipschitz classification of complex algebraic plane curves at infinity. 
 In order to present our result concerned to that classification we need to introduce more definitions and notations. We denote by $\P^2$ the projective plane.
 Let $[x:y:z]\in\P^2$ denote the subspace spanned by $(x,y,z)$, and let $\iota: \C^2\hookrightarrow \P^2$ be the parametrization given by $\iota(x,y)=[x:y:1]$. 
The \textbf{line at infinity}, denoted by $L_{\infty}$, is the complement of $\iota( \mathbb{C}^2)$ in  $\mathbb{P}^2 $.

\begin{definition}
 Let $f\in \C[x,y]$ be a polynomial of degree $n$.
 The \textbf{homogenization} of $f$ is the homogeneous  polynomial $\widetilde{f}\in\C[x,y,z]$ defined by 
 $$\widetilde{f}(x,y,z)=z^nf\left(\frac{x}{z},\frac{y}{z}\right).$$
 Let $C$ be a complex algebraic plane curve with equation   $f(x,y)=0$. 
 The  curve $\widetilde{C}=\{[x:y:z]\in \P^2:\widetilde{f}(x,y,z)=0\}$ in the projective plane is called the \textbf{homogenization} of $C$. 
 The \textbf{points at infinity} of $C$ are the elements of the intersection $\widetilde{C}\cap L_{\infty}$. 
 \end{definition}

 We prove that the Lipschitz geometry at infinity of a complex algebraic plane curve $C$ determines and is determined by the topological type of the germ of the curve $\widetilde{C}\cup L_{\infty}$ at each point at infinity of $C$.
 Since the topological type of germs of complex plane curves are presented in terms of dual resolution graphs we also encode  the Lipschitz geometry at infinity in a tree obtained as a quotient of dual resolution graphs as follows.
 
 \begin{theorem}\label{atinfinity}
Let $C$ and $C'$ be two complex algebraic plane curves.
The following are equivalent:
 \begin{enumerate}
      \item\label{it1} $C$ and $C'$ have the same Lipschitz geometry at infinity;
  \item\label{it2} there is a bijection $\psi$ between the set of points at infinity of $C$ and the set of points at infinity of $C'$ such that $(\widetilde{C}\cup L_\infty,p)$ has the same topological type as $(\widetilde{C}'\cup L_\infty,\psi(p))$;
  \item\label{it3} there is an isomorphism between the Lipschitz tree at infinity of  $C$ and $C'$ (see definition \ref{liptree}).
  \end{enumerate}
\end{theorem}
 Armed with the classification of the Lipschitz geometry of germs and of the Lipschitz geometry at infinity of complex algebraic plane curves we obtain our main result.

 \begin{theorem}\label{global}
Let $C$ and $\Gamma$ be two   complex plane algebraic  curves with irreducible components $ C= \bigcup_{i\in I} C_i$ and $\Gamma=\bigcup_{j\in J} \Gamma_j$. The following are equivalent:
 
   \begin{enumerate}
     \item \label{iti}$C$ and $\Gamma$ have the same Lipschitz geometry;
     \item \label{itii} there are bijections $\sigma: I\to J$ and $\varphi$ between the set of singular points  of $\widetilde{C}\cup L_\infty$ and the set of singular points  of $\widetilde{\Gamma}\cup L_\infty$ such that $p\in L_\infty$ if only if $\varphi(p)\in L_\infty$,  $(\widetilde{C}\cup L_\infty,p)$ has the same topological type as $(\widetilde{\Gamma}\cup L_\infty,\varphi(p))$, and each $(\widetilde{C}_i\cup L_\infty,p)$ has the same topological type as $(\widetilde{\Gamma}_{\sigma(i)}\cup L_\infty,\varphi(p))$;
     \item \label{itiii} there is an isomorphism between the Lipschitz graph of  $C$ and $\Gamma$ (see definition \ref{Lipgraph}).
 \end{enumerate}
\end{theorem}

We organize the paper in the following way. In Section \ref{trees}, we present definitions of Eggers-Wall and carousel tree. We also describe how one gets the Eggers-Wall tree from the carousel tree.

 Section \ref{lipimplies} is devoted to prove that the Lipschitz geometry at infinity of a complex plane algebraic  curves gives us the topological data which implies (\ref{it2}) of Theorem \ref{atinfinity}, i.e., we prove that (\ref{it1}) implies (\ref{it2}).
 We also give the definition of Lipschitz graph at infinity and explain the equivalence between (\ref{it2}) and (\ref{it3}).
In Section \ref{topsimplies}, we  prove that (\ref{it2}) implies (\ref{it1}) of Theorem \ref{atinfinity}.

In the last section, we define Lipschitz graph of complex plane algebraic  curves and  prove Theorem \ref{global}.

\subsection*{Acknowledgments.} 
I would like to thank Edson Sampaio, Lev Birbrair and Rodrigo Mendes for valuable discussions on the subject. 
I am deeply indebted to Alexandre Fernandes and Anne Pichon for supervise me along this work which is part of my PhD thesis.

This work has been partially supported by CAPES/COFECUB  project 88887.\\ 143177/2017-00 - An\'alise Geom\'etrica e Teoria de Singularidade em Espa\c cos Estratificados and  by the project Lipschitz geometry of singularities (LISA) of the Agence Nationale de la Recherche (project ANR-17-CE40-0023) and also by Instituto Federal de Educa\c c\~ao, Ci\^encia e Tecnologia do Cear\'a (IFCE).

\section{Plane curve germs  and their Eggers-Wall and carousel  trees }\label{trees}
In this section we explain the basic notations and conventions used throughout the paper about reduced germs $C$ of complex curves on smooth surfaces. 
Then we define the Eggers-Wall tree and the carousel tree of such a germ relative to a smooth branch contained in it. 
The definition of Eggers-Wall tree which are given in this paper are the same present in \cite{GarciaBarroso2019}.  
Finally, we  describe how one gets the Eggers-Wall tree from the carousel tree. This process is also described in \cite{Anne}.

We  recall some definitions  and conventions about power series with positive rational exponents. Let $n$ be a positive integer, the ring $\C[[x^{1/n}]]$ consists of sequence $ (A_k)_{k\in \N}$ of elements of $\C$.  
Let $\eta=(A_k)_{k\in \N}\in \C[[x^{1/n}]]$, we denote this element by
$$\eta=\sum_{k=0}^{\infty}A_kx^{k/n}.$$
 The \textbf{exponents} of $\eta$ are the numbers $k/n$ such that $A_k\neq0$. We denote the set of exponents of $\eta$ by $\mathcal{E}(\eta)$.  The \textbf{order} of $\eta\neq 0$, denoted by $\ord_x \eta$, is the smallest exponent of $\eta$. For technical reasons it is convenient to define the order of the zero to be $+\infty$. The subgroup of $n$-th roots of 1 acts on $\C[[x^{1/n}]]$ by the rule
$$(\rho,\eta)\to \eta(\rho\cdot x^{1/n}):=\sum_{k=0}^{\infty}A_k\rho^kx^{k/n},\text{ where  $\rho$ is a $n$-th root of 1.}$$

All over this section, $\mathcal{S}$  denotes a complex manifold of dimension two. 
We fix a point $O\in \mathcal{S}$. 
All coordinate charts of this section are defined in a neighborhood of $O$, moreover, the point $O$ always has coordinate $(0,0)\in\C^2$. A curve  germ in $(\mathcal{S},O)$ is the zero set of a non-constant holomorphic function germ from $(\mathcal{S},O)$ to $(\C,0).$      We denote by $(C,O)$ the germ of $C$ at $O$ and by $\mathcal{O}_O$ the ring of holomorphic function germs at $O$.

Any chart of $\mathcal{S}$  induces an isomorphism between $\mathcal{O}_O$ and $\C\{x,y\}$.   Since $\C\{x,y\}$ is factorial, $\mathcal{O}_O$ is factorial. Let $C$ be   a complex curve with equation $f=0$. 
Then $f$ can be written as a product $g_1^{\alpha_1}\ldots g_k^{\alpha_k}$, with $g_1,\ldots, g_k$ irreducible, and the $\alpha_j$'s are positive integers.
The zero set of $g_j$'s are the \textbf{branches} of $C$. When $k=1$, we say that $C$ is \textbf{irreducible}. 
The holomorphic function  $f$ is \textbf{reduced} if each $\alpha_j=1$. We will always suppose all equations for curves are reduced.
The curve $C$ is said to be \textbf{smooth at} $O$   if  there is a neighborhood $U$ of $O$ in $\mathcal{S}$ such that $C\cap U$ is a complex submanifold of $U$.

The next definitions of this section depend on the choice of a smooth curve $L$ at $O$. 
In this section, we always choose a coordinate system $(x,y)$ such that $L=\{x=0\}$. Assume that a coordinate system $(x,y)$ is fixed. 
Let $C$ be a  curve on $\mathcal{S}$  and assume that $A$ is a branch of $C$  different from the  curve $L$. Relative to the system $(x,y)$, the branch $A$ may be defined by a Weierstrass polynomial $f_A\in \C\{x\}[y]$, which is monic, and of degree $d_A$.
Note that the degree $d_A$ does not depend on the system of coordinates.  

By the Newton-Puiseux Theorem, there exists a parametrization of $A$ of the form $\gamma_A(w)=(w^{d_A},\eta_A(w))$ where $\eta_A(w)=\sum_{k>0}a_kw^k\in \C\{w\}$. 
Let $n$ be the product of the degrees of the Weierstrass polynomials of the branches of $C$ different from $L$.
We consider  the formal power series  $\sum_{k=0}A_kx^{k/n}\in \C[[x^{1/n}]]$ where 
$$A_k=\begin{cases}
			a_\frac{kd_A}{n}, & \text{if $n$ divides $kd_A$ }\\
            0, & \text{otherwise.}
		 \end{cases}$$
We still denote by $\eta_A$ the formal power series $\sum_{k=0}A_kx^{k/n}$. 
 The \textbf{Newton-Puiseux roots relative to} $L$ of the branch $A$ are the formal power series $\eta_A(\rho\cdot x^{1/n})\in \C[[x^{1/n}]]$, for $\rho$ running through the $n$-th roots of 1.

 Let $\rho\in\C$ be a primitive $n$-root of unity, notice that there are only $d_A$ Newton-Puiseux roots relative to $L$ of the branch $A$, namely $$\eta_A(\rho\cdot x^{1/n}),\ldots, \eta_A(\rho^{d_A}\cdot x^{1/n}).$$

All the Newton-Puiseux roots relative to $L$ of the curve $A$ have the same exponents. Some of those exponents may be distinguished by looking at the differences of roots:

\begin{definition}
The \textbf{characteristic exponents relative to} $L$ of the curve $A$ are the $x$-orders $\ord_x (\eta_A-\eta'_A)$ of the differences between distinct Newton-Puiseux roots relative to $L$ of $A$.
\end{definition}

The characteristic exponents relative to $L$ of $A$ consist of exponents of $\eta_A$ which, when written as a quotient of integers, need a denominator strictly bigger than the lowest common denominator of the previous exponents.  That is: $\frac{l}{n}$ is 
characteristic exponent relative to $L$ of $A$ if and only if $N_l\frac{l}{n}\not\in\Z$ where $N_l=\min\{N\in \Z\ ;\mathcal{E}(\eta_A)\cap[0,\frac{l}{n})\in\frac{1}{N}\Z \} $.

By \cite[Proposition 3.10]{GarciaBarroso2019}
 the characteristic exponents relative to $L$ do not depend on the coordinate system
$(x, y)$, but only on the branch $L$.

The \textbf{Newton-Puiseux roots relative} to $L$ of the curve $C$ are the Newton-Puiseux roots relative to $L$ of its branches different from $L$. Let us denote by $\mathcal{I}_C$ the set of
branches of $C$ which are different from $L$. Therefore, $C$ has $d_C:=\sum_{A\in \mathcal{I}_C} d_A$ Newton-Puiseux roots relative to $L$. 

\begin{example}\label{examplelocal}
Let $L$ be the $y$-axis. Consider a plane curve $C$ whose branches $A$ and $B$ are  parametrized by 
$$\gamma_A(w)=(w^4,w^6+w^7),\, \gamma_B(w)=(w^2,w),  $$
respectively. The Newton-Puiseux roots relative to $L$ of $A$ are 
\begin{align*}
\eta_A(x^{1/8})& =x^{12/8}+x^{14/8},&
\eta_A(\rho x^{1/8})& =\rho^4x^{12/8}+\rho^6x^{14/8},\\
\eta_A(\rho^2 x^{1/8})&=x^{12/8}+\rho^4x^{14/8},&
\eta_A(\rho^3 x^{1/8})&=\rho^4x^{12/8}+\rho^2x^{14/8},
\end{align*}
where $\rho$ is a primitive 8-th root of unity. While the Newton-Puiseux roots relative to $L$ of $B$ are
\begin{align*}
\eta_B(x^{1/8})& =x^{4/8},&
\eta_B(\rho x^{1/8})& =\rho^4x^{4/8}.
\end{align*}
The characteristic exponents relative to $y$-axis of $A$ are $3/2, 7/4$. The characteristic exponent of $B$ relative to $y$-axis is $1/2$.
\end{example}

We keep assuming that $A$ is a branch of $C$ different from $L$.
The Eggers-Wall tree of $A$ relative to $L$ is a geometrical way of encoding the set of characteristic exponents, as well as
the sequence of their successive common denominators:

\begin{definition}
 The Eggers-Wall tree $\Theta_L (A)$ of the curve $A$ relative to $L$ is a
compact oriented segment endowed with the following supplementary structures:
\begin{itemize}
    \item  an increasing homeomorphism $\mathbf{e}_{L,A} : \Theta_L (A) \to [0,\infty]$, \textbf{the exponent function};
\item \textbf{marked points}, which are by definition the points whose values by the exponent function
are the characteristic exponents of $A$, as well as the smallest end of $\Theta_L (A)$,
labeled by $L$, and the greatest end, labeled by $A$.
\item an \textbf{index function} $\mathbf{i}_{L,A} : \Theta_L (A) \to \N$, which associates to each point $P \in \Theta_L (A)$ the smallest common denominator of the exponents
of a Newton-Puiseux root of $A$ which are strictly less than $\mathbf{e}_{L,A}(P)$.
\end{itemize}
\end{definition}
Let us consider now the case of a  curve with several branches. In order to construct the Eggers-Wall tree in this case, one needs to know not only 
the characteristic exponents of its branches, but also the \emph{exponent  of coincidence} 
of its pairs of branches:

\begin{definition}
If $A$ and $B$ are two distinct branches of $C$, then their \textbf{exponent of coincidence relative} to $L$ is defined by:
$$k_L (A, B) := \max\{\ord_x (\eta_A - \eta_B) \}, $$
where $\eta_A,\eta_B\in\C[[x^{1/n}]]$ vary among the Newton-Puiseux roots of $A$ and $B$, respectively.  
\end{definition}

\begin{definition}
Let $C$ be a germ of curve on $(S,O)$. 
Let us denote by $\mathcal{I}_C$ the set of
branches of $C$ which are different from $L$. The Eggers-Wall tree $\Theta_L (C)$ of $C$
relative to $L$ is the rooted tree obtained as the quotient of the disjoint union of the individual Eggers-Wall trees $\Theta_L (A), A \in \mathcal{I}_C$, by the following equivalence relation. If $A, B \in\mathcal{I}_C$,
then we glue  $\Theta_L (A)$ with $\Theta_L (B)$ along the initial segments $\mathbf{e}^{-1}_{L,A}([0, k_L (A, B)])$
and $\mathbf{e}^{-1}_{L,B}([0, k_L (A, B)])$ by:
$$\mathbf{e}^{-1}_{L,A}(\alpha) \sim \mathbf{e}^{-1}_{L,B}(\alpha),\text{ for all } \alpha \in [0, k_L (A, B)].$$
One endows $\Theta_L (C)$ with the \textbf{exponent function} $\mathbf{e}_L : \Theta_L (C) \to [0,\infty]$ and the index function
$\mathbf{i}_L : \Theta_L (C) \to \N$ induced by the initial exponent functions $\mathbf{e}_{L,A}$ and $\mathbf{i}_{L,A}$ respectively,
for $A$ varying among the irreducible components of $C$ different from $L$. The tree $\Theta_L (L)$ is the
trivial tree with vertex set a singleton whose element is labelled by $L$. If $L$ is an irreducible
component of $C$, then the marked point $L \in \Theta_L (L)$ is identified with the root of $\Theta_L (L)$ for
any $A \in \mathcal{I}_C$. The set of marked points of $\Theta_L (C)$ is the union of the set of marked points of the
Eggers-Wall tree of the branches of $C$ and of the set of ramification points of $\Theta_L (C)$.
\end{definition}

Again, the fact that in the previous notations $\Theta_L (C), \mathbf{e}_L , \mathbf{i}_L$ we mentioned only the dependency on $L$, and not on the  coordinate system $(x, y)$, comes from  \cite[Proposition 3.10]{GarciaBarroso2019}.
\begin{example}
Consider again the curve of Example \ref{examplelocal}. One has $K_L(A,B)=1/2$ and the Eggers-Wall tree of $C$ relative to $L$ is drawn in Figure \ref{fig:figure1}.

\begin{figure}[h]
	\centering
\begin{tikzpicture}
\draw[thin ](0,0)--(0,4);
\draw[thin ](0,1)--(2,2.5);

\draw[fill ] (0,0)circle(2pt);
\draw[fill ] (0,1)circle(2pt);
\draw[fill ] (0,2)circle(2pt);
\draw[fill ] (0,3)circle(2pt);
\draw[fill ] (0,4)circle(2pt);
\draw[fill ] (2,2.5)circle(2pt);

\node(a)at(-0.3,0){$0$};
\node(a)at(-0.3,1){$\frac{1}{2}$};
\node(a)at(-0.3,2){$\frac{3}{2}$};
\node(a)at(-0.3,3){$\frac{7}{4}$};

\node(a)at(0.2,0.5){$1$};
\node(a)at(0.2,1.5){$1$};
\node(a)at(0.2,2.5){$2$};
\node(a)at(0.2,3.5){$4$};
\node(a)at(1,2){$2$};

\node(a)at(0.3,-0.3){$L$};
\node(a)at(0.3,4.3){$A$};
\node(a)at(2.3,2.5){$B$};

	\end{tikzpicture}
	\caption{Eggers-Wall tree or Example \ref{examplelocal}.}\label{fig:figure1}
\end{figure}
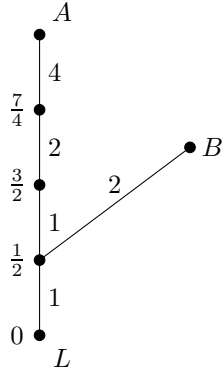

\end{example}
The carousel tree is a  variant of the Eggers-Wall tree, but using all the Newton-Puiseux
roots of $C$, not only one root for each branch. The name was introduced in \cite{Anne} and it is  inspired by the carousel geometrical model for the link of the curve $C$ described in \cite[Section 5.3]{Wall}.

\begin{definition}\label{carouseultree}
Let $C$ be a germ of curve on $\mathcal{S}$. 
Let us denote by $[d_C]$ the set $\{1,\ldots,d_C\}$  and let $\eta_j,j\in[d_C]$ be the Newton-Puiseux roots relative to $L$ of $C$. Consider the map $\ord_x\colon
  [d_C]\times[d_C]\to \Q\cup\{\infty\}$, $(j,k)\mapsto \ord_x(\eta_j-\eta_k)$. 
    The map $\ord_x$ has the property that $\ord_x(j,l) \geq \min \{\ord_x(j,k),\ord_x(k,l)\}$ for any triple $j,k,l$.   So for any $q\in \Q\cup\{\infty\}$, the relation on the set $[d_C]$ given by $j\sim_q k\Leftrightarrow \ord_x(j,k)\ge q$ is an equivalence relation.
    Name the elements of the set $\ord_x([d_C]\times[d_C])\cup\{0\}$ in ascending order: $0=q_0<q_1<\dots<q_r=\infty$.
  For each $i=0,\dots,r$ let $G_{i,1},\dots,G_{i,\mu_i}$ be the equivalence classes for the relation $\sim_{q_i}$.
  So $\mu_r=d_C$ and the sets $G_{r,j}$ are singletons while $\mu_0=\mu_1=1$ and $G_{0,1}=G_{1,1}=[d_C]$.
  We form a tree with these equivalence classes $G_{i,j}$ as vertices and edges given by inclusion relations: there is an edge between $G_{i,j}$ and $G_{i+1,k}$ if $G_{i+1,k}\subseteq G_{i,j}$. 
  The vertex $G_{0,1}$ is the root of this tree and the singleton sets $G_{r,j}$ are the leaves.
  We weight each vertex with its corresponding $q_i$. The \textbf{carousel tree relative} to $L$ is the tree obtained from this tree by suppressing valency 2 vertices:  we remove each such vertex and amalgamate its two adjacent edges into one edge. 
  
\end{definition}

   We will describe how one gets the Eggers-Wall tree from the carousel tree. This process is essentially the same process described in \cite[Lemma 3.1]{Anne}.
   At any vertex $v$ of the carousel tree we have a weight $q_v$ which is one of the  $q_i$'s. 
   Let $d_v$ be the denominator of the $q_v$ when $q_v$ is written as a quotient of coprime integers.
  
  The process of obtaining the Eggers-Wall tree from the carousel tree is  an induction process in $i$.
  First, we label the edge between $G_{0,1}$ and $G_{1,1}$ by 1.
  The subtrees cut off above $G_{1,1}$ consist of groups of $d_{G_{1,1}}$ isomorphic trees, with possibly one additional tree.
  We label the edge connecting $G_{1,1}$ to this additional tree, if it exists, with $1$,  and then delete all but one from each group of $d_{G_{1,1}}$ isomorphic trees. 
  Finally, we label the remaining edges contain $G_{1,1}$  with $\lcm\{d_{G_{1,1}},1\}$.
  
  Inductively, let $v$ be a vertex with weight $q_i$.  Let $v'$ be the adjacent vertex below $v$ along the path from $v$ up to the root vertex and let $l_{vv'}$ be the label of the edge between $v$ and $v'$. 
The subtrees cut off above $v$ consist of groups of $\frac{\lcm\{d_{v},l_{vv'}\}}{l_{vv'}} $ isomorphic trees, with possibly one additional tree.
We label the edge connecting $v$ to this additional tree, if it exists, with  $l_{vv'}$,  and then delete all but one from each group of $\frac{\lcm\{d_{v},l_{vv'}\}}{l_{vv'}} $   isomorphic trees below $v$.  
 Finally, we label the remaining edges contain $v$ with $\lcm\{d_{v},l_{vv'}\}$.

The resulting tree, with the $q_v$ labels at vertices and the extra label on the edges is easily recognized
as the Eggers-Wall tree relative to $L$ of $C$.

\begin{example}
Figure \ref{fig:figure2}  illustrates the above process for the Example \ref{examplelocal}.

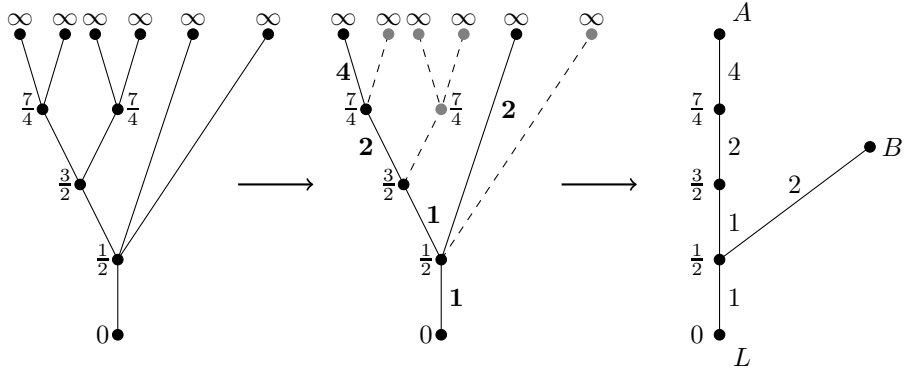
\begin{figure}[h]
	\centering
	\begin{tikzpicture}

\draw[thin ](0,0)--(0,1);
\draw[fill ] (0,0)circle(2pt);
\draw[fill ] (0,1)circle(2pt);

\draw[thin ](0,1)--(1,4);
\draw[fill ] (1,4)circle(2pt);
\node(a)at(1,4.2){$\infty$};

\draw[thin ](0,1)--(2,4);
\draw[fill ] (2,4)circle(2pt);
\node(a)at(2,4.2){$\infty$};

\draw[thin ](0,1)--(-0.5,2);
\draw[thin ](-0.5,2)--(0,3);
\draw[thin ](-0.5,2)--(-1,3);

\draw[fill ] (-0.5,2)circle(2pt);
\draw[fill ] (0,3)circle(2pt);
\node(a)at(0.2,3){$\frac{7}{4}$};

\node(a)at(-1.2,3){$\frac{7}{4}$};

\node(a)at(-0.7,2){$\frac{3}{2}$};

\node(a)at(-0.2,1){$\frac{1}{2}$};

\node(a)at(-0.2,0){$0$};

\draw[thin ](-1.3,4)--(-1,3);
\draw[fill ] (-1.3,4)circle(2pt);
\draw[fill ] (-1,3)circle(2pt);
\node(a)at(-1.3,4.2){$\infty$};

\draw[thin ](-0.7,4)--(-1,3);
\draw[fill ] (-0.7,4)circle(2pt);
\node(a)at(-0.7,4.2){$\infty$};

\draw[thin ](-0.3,4)--(0,3);
\draw[fill ] (-0.3,4)circle(2pt);
\node(a)at(-0.3,4.2){$\infty$};

\draw[thin ](0.3,4)--(0,3);
\draw[fill ] (0.3,4)circle(2pt);
\node(a)at(0.3,4.2){$\infty$};

\draw[thick,>-stealth,->](1.6,2)--+(1.0,0);

\begin{scope}[xshift=4.3cm]
\draw[thin ](0,0)--(0,1);
\draw[fill ] (0,0)circle(2pt);
\draw[fill ] (0,1)circle(2pt);
\node(a)at(-0.2,0){$0$};

\draw[thin ](0,1)--(1,4);
\draw[fill ] (1,4)circle(2pt);
\node(a)at(1,4.2){$\infty$};

\draw[dashed ](0,1)--(2,4);
\draw[fill=gray, color=gray ] (2,4)circle(2pt);
\node(a)at(2,4.2){$\infty$};

\draw[thin ](0,1)--(-0.5,2);

\draw[dashed ](-0.5,2)--(0,3);
\draw[thin ](-0.5,2)--(-1,3);

\draw[fill ] (-0.5,2)circle(2pt);
\draw[fill=gray, color=gray ] (0,3)circle(2pt);
\node(a)at(0.2,3){$\frac{7}{4}$};

\node(a)at(-1.2,3){$\frac{7}{4}$};

\node(a)at(-0.7,2){$\frac{3}{2}$};

\node(a)at(-0.2,1){$\frac{1}{2}$};

\draw[thin ](-1.3,4)--(-1,3);
\draw[fill ] (-1.3,4)circle(2pt);
\draw[fill ] (-1,3)circle(2pt);
\node(a)at(-1.3,4.2){$\infty$};
\node(a)at(-1.3,3.5){$\mathbf 4$};

\draw[dashed ](-0.7,4)--(-1,3);
\draw[ fill=gray, color=gray ] (-0.7,4)circle(2pt);
\node(a)at(-0.7,4.2){$\infty$};

\draw[dashed ](-0.3,4)--(0,3);
\draw[fill=gray, color=gray ] (-0.3,4)circle(2pt);
\node(a)at(-0.3,4.2){$\infty$};

\draw[dashed ](0.3,4)--(0,3);
\draw[fill=gray, color=gray ] (0.3,4)circle(2pt);
\node(a)at(0.3,4.2){$\infty$};

\node(a)at(-1,2.5){$\mathbf 2$};

\node(a)at(0.9,3){$\mathbf 2$};
\node(a)at(-0.1,1.6){$\mathbf 1$};

\node(a)at(0.2,0.5){$\mathbf 1$};

\draw[thick,>-stealth,->](1.6,2)--+(1,0);

\end{scope}

\begin{scope}[xshift=8cm]
\draw[thin ](0,0)--(0,4);
\draw[thin ](0,1)--(2,2.5);

\draw[fill ] (0,0)circle(2pt);
\draw[fill ] (0,1)circle(2pt);
\draw[fill ] (0,2)circle(2pt);
\draw[fill ] (0,3)circle(2pt);
\draw[fill ] (0,4)circle(2pt);
\draw[fill ] (2,2.5)circle(2pt);

\node(a)at(-0.3,0){$0$};
\node(a)at(-0.3,1){$\frac{1}{2}$};
\node(a)at(-0.3,2){$\frac{3}{2}$};
\node(a)at(-0.3,3){$\frac{7}{4}$};

\node(a)at(0.2,0.5){$1$};
\node(a)at(0.2,1.5){$1$};
\node(a)at(0.2,2.5){$2$};
\node(a)at(0.2,3.5){$4$};
\node(a)at(1,2){$2$};

\node(a)at(0.3,-0.3){$L$};
\node(a)at(0.3,4.3){$A$};
\node(a)at(2.3,2.5){$B$};

\end{scope}

	\end{tikzpicture}
	\caption{From the carousel tree to the Eggers-Wall tree. }\label{fig:figure2}
\end{figure}	   
\end{example}

\section{Lipschitz geometry at infinity determines topological type}\label{lipimplies}

In this section, we define the Lipschitz tree at infinity of a complex algebraic plane curve. Then we prove the equivalence of (\ref{it2}) and (\ref{it3}) and  that (\ref{it1}) implies (\ref{it2}) of Theorem \ref{atinfinity}. 

To define the Lipschitz tree at infinity of a complex algebraic plane curve we recall the basic vocabulary of resolution of singularities. Let $(C,p)\subset (\mathcal{S},p)$ be a germ of a singular complex curve in a smooth surface $\mathcal{S}$.
We remember that the blowing up of $\mathcal{S}$ with centre $p$ produces a smooth surface $\mathcal{S}_1$, a holomorphic map  $\pi_1:\mathcal{S}_1\to\mathcal{S}$ such that $\pi_1:\mathcal{S}_1\backslash\pi_1^{-1}(p)\to S\backslash \{p\}$ is biholomorphic, the \textbf{exceptional curve} $E_1=\pi_1^{-1}(p)$, and the \textbf{strict transform} $C_1$ which is the topological closure  $\overline{\pi_1^{-1}(C\backslash\{p\})}$.  
The map $\pi_1$ is called the \textbf{blowing up} of $\mathcal{S}$ with centre $p$.
A \textbf{good minimal resolution} of $C$ is a map $\pi:\mathcal{S}_n\to\mathcal{S}$ which is a composite of finite and minimal sequence of  blowing ups $\pi_i:\mathcal{S}_i\to\mathcal{S}_{i-1}$ such that the strict transform $C_n=\overline{\pi^{-1}(C\backslash\{p\})}$ is smooth and meets the exceptional curves $\pi^{-1}(p)=E_1\cup E_2\cup\cdots\cup E_n$ transversely at  regular points.

\begin{definition}\label{liptree}
Let $C$ be a complex algebraic plane curve, $p_1,\ldots,p_m$ its points at infinity and let $B_1^{(j)},\ldots,B_{k_j}^{(j)}$ be the branches of $(\widetilde{C},p_j)$.
A \textbf{good minimal resolution} of $(\widetilde{C}\cup L_\infty,p_1)$ produces a smooth surface $S_{(1)}$, a projection $\pi_{(1)}:S_{(1)}\to \P$, a sequence of exceptional curves $E_1^{(1)},\ldots,E_{r_1}^{(1)}$ and strict transform curves $\mathcal{B}_1^{(1)},\ldots, \mathcal{B}_{k_1}^{(1)}$ of the branches $B_1^{(1)},\ldots, B_{k_1}^{(1)}$ and the strict transform $\mathcal{L}_\infty$ of the line at infinity $L_\infty$.
Then, we resolve the strict transform $C^{(1)}=\pi^{-1}_{(1)}(\widetilde{C})$ at the singular point $\pi_{(1)}^{-1}(p_2)$.
We repeat this process for all points at infinity of $C$. 

The \textbf{Lipschitz tree at infinity} of $C$  is rooted tree with vertices $V^{(j)}_k$ corresponding to the curves $E^{(j)}_k$ labeled with its  self-intersection number,  arrow vertices $W^{(j)}_i$ corresponding to the branches  $\mathcal{B}_i^{(j)}$ not labeled and a root corresponding to the strict transform $\mathcal{L}_\infty$ of the line at infinity. 
We put an edge joining vertices if and only if the corresponding curves intersect each other.
\end{definition}

   \begin{example}\label{paracusp}
 The Lipschitz tree at infinity  of the complex algebraic plane curve defined by $(y-x^2)(y^3-x)=0$ is drawn in Figure \ref{fig:Liptree}.
 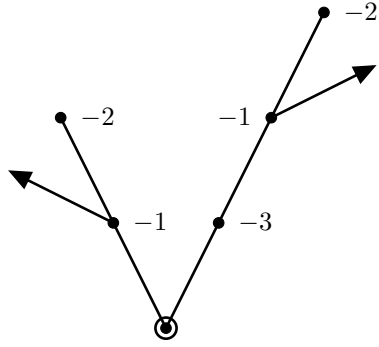
\begin{figure}[h]
     \centering
 \begin{tikzpicture}[line cap=round,line join=round,>=triangle 45,x=0.7cm,y=0.7cm]
\clip(-6,-1) rectangle (6,7);
\draw [line width=1pt] (0,0)-- (-1,2);
\draw [line width=1pt] (0,0)-- (1,2);
\draw [line width=1pt] (-1,2)-- (-2,4);
\draw [->,line width=1pt] (-1,2)-- (-3,3);
\draw [line width=1pt] (1,2)-- (2,4);
\draw [line width=1pt] (2,4)-- (3,6);
\draw [->,line width=1pt] (2,4)-- (4,5);
\draw [line width=1pt] (0,0) circle (4pt);
\begin{scriptsize}
\draw [fill=black] (0,0) circle (2pt);
\draw [fill=black] (-1,2) circle (2pt);
\draw[color=black] (-0.3,2) node {\normalsize $-1$};
\draw [fill=black] (1,2) circle (2pt);
\draw[color=black] (1.7,2) node {\normalsize $-3$};
\draw [fill=black] (-2,4) circle (2pt);
\draw[color=black] (-1.3,4) node {\normalsize $-2$};
\draw [fill=black] (2,4) circle (2pt);
\draw[color=black] (1.3,4) node {\normalsize $-1$};
\draw [fill=black] (3,6) circle (2pt);
\draw[color=black] (3.7,6) node {\normalsize $-2$};
\end{scriptsize}
\end{tikzpicture}
\caption{Lipschitz tree at infinity of Example \ref{paracusp}. }
     \label{fig:Liptree}
\end{figure}
\end{example}

We point out that the Lipschitz tree at infinity of $C$ is obtained as the quotient of the disjoint union of the individual dual resolution graph of minimal good resolutions of $(\widetilde{C}\cup L_\infty,p_i)$, by identifying
all vertices corresponding to the strict transform of the line at infinity and putting it as the root. 
We recall that by \cite[Theorem 8.1.7]{Wall}, the isomorphism class of the  dual resolution graph of a minimal good resolutions of  germ  of complex curve at singular point   determines and it is determined by its topological type.
This explain the equivalence $(\ref{it2})\Leftrightarrow (\ref{it3})$ of Theorem \ref{atinfinity}.

For the implication $(1)\Rightarrow (2)$ of Theorem \ref{atinfinity}, we introduce the asymptotic notations of Bachmann-Landau  which are convenient for  study of Lipschitz geometry (see \cite{knuth76bigomicron} for a historical survey about these notation).

\begin{definition}
Let $f,g:(0,\infty)\to (0,\infty)$ be two functions. 
We say  that
\begin{enumerate}
    \item \textbf{$f$ is big-Theta of $g$}, and we write  $f(t)=\Theta(g(t))$, if there exists $R_0>0$ and a constant  $c>0$  such that $\displaystyle\frac{1}{c}g(t)\leq f(t)\leq cg(t)$ for all  $t>R_0$.

    \item \textbf{$f$ is small-o of $g$}, and we write  $f(t)=o(g(t))$, if  $\displaystyle \limsup\limits_{t\rightarrow \infty}\frac{f(t)}{g(t)}=0$.
\end{enumerate}
\end{definition}
  
Let $[a:b:0]$ be a point at infinity  of a complex algebraic plane curve $C$.
The linear subspace spanned by $(a,b)$ in $\C^2$ is the \textbf{tangent line at infinity} to $C$  associated with $[a:b:0]$ (see \cite{Fernandes2019} and \cite{CTinfinity}).

\begin{example}\label{example}
Consider the polynomial $f(x,y)= y^2x-y$, and let $C_\lambda$ be the  complex algebraic plane curve with equation $f(x,y)+\lambda=0$ for $\lambda\in\C$. One has $\widetilde{f}(x,y,z)=y^2x-yz^2+\lambda z^3$, and the points at infinity of $C_\lambda$ are $[1:0:0]$ and $[0:1:0]$ and the tangent lines at infinity to $C_\lambda$ are the coordinates axis. 
\end{example}

\begin{lemma}\label{semnome}
Let $C$ be a complex algebraic plane curve, and let $P:\C^2\to \C$ be a linear projection  whose kernel does not contain any tangent line   at infinity to $C$.
Then there exist a compact set $K$ and  a constant $M>1$ such that for each $u,u' \in C\backslash K $, there is an arc $\tilde \alpha$ in  $C\backslash K$ joining $u$ to a point $u''$ with $P(u'')=P(u')$ and
$$ d(u,u') \leq \len(\tilde \alpha) + d(u'',u') \leq M d(u,u').$$ 

\end{lemma}

\begin{proof}
After a linear change of coordinates if necessary, we may assume  that $P$ is the projection on the first coordinate  and that the $y$-axis is not a tangent line at infinity to $C$.  
Let $[1:a_1:0],\ldots,[1:a_m:0]$ be the points at infinity of $C$. For each $i$, let  $B_{i1},\ldots,B_{ik_i} $ be the branches of $(\widetilde{C},[1:a_i:0])$.

The open set $U=\{[x:y:z]\in \P^2:x\neq0\}$ contains all the points at infinity of $C$, so we can use   the coordinate chart $\varphi:U\to\C^2$  defined by $\varphi([x:y:z])=(z/x,y/x)$ to obtain Newton-Puiseux parametrization of the branch $\varphi(B_{ij})$ for each $i$. 
Let  $\epsilon>0$ sufficiently small such that there  exists Newton-Puiseux parametrization $\gamma_{ij}:D_\epsilon\to\C^2$ of $\varphi(B_{ij})$ given by
$$\gamma_{ij}(w)=(w^{d_{ij}},a_i+v_{ij}(w)),$$
where $D_\epsilon$ is the open disk of radius $\epsilon$ centered at the origin and $v_{ij}\in\C\{w\}, v_{ij}(0)=0$. 
Let $\Gamma_{ij}:D_\epsilon\backslash\{0\}\to\C^2$ given by $$\Gamma_{ij}(w)=(\iota^{-1}\circ\varphi^{-1}\circ\gamma_{ij})(w)=\left(\frac{1}{w^{d_{ij}}},\frac{a_i+v_{ij}(w)}{w^{d_{ij}}}\right).$$

 We will prove that  the compact $K=C\backslash \bigcup_{ij} \Gamma_{ij}(D_\epsilon\backslash\{0\})$ satisfies the desired conditions.

We claim that there exists a constant $c>0$  such that $C\backslash K$ is a subset of the cone $\{(x,y)\in \C^2;|y|\leq c|x|\}$. Moreover, $c$ may be chosen such that the tangent space of $C\backslash K$ at a point $p$, denoted by $T_pC$, is also a subset of the same cone.

The first part of this statement is easy to check. In particular, it follows that  $P|_{\Gamma_{ij}(D_\epsilon\backslash\{0\})}$ is a covering map for all $i,j$.
 Differentiating $\Gamma_{ij}$ gives 
$$\Gamma'_{ij}(w)=\left(-\frac{d_{ij}}{w^{d_{ij}+1}},\frac{wv'_{ij}(w)-d_{ij}v_{ij}(w)}{w^{d_{ij}+1}}-a_i\frac{d_{ij}}{w^{d_{ij}+1}}\right).$$

 Thus the points $(x,y)\in T_{\Gamma_{ij}(w)}C$  satisfies $|y-a_ix|\leq\eta_{ij} |x|\Rightarrow |y|\leq(\eta_{ij}+|a_i|) |x|$ where $\eta_{ij}=\sup\left|\frac{wv'_{ij}(w)-d_{ij}v_{ij}(w)}{d_{ij}}\right|$. 
 Now, putting $c=\max_{ij}\{\eta_{ij}+|a_i|\}$ we have 
$$T_p C \subset \{(x,y)\in \C^2;|y|\leq c|x|\} \text{ for all } p\in C\backslash K,$$
as claimed.

Suppose $u,u' \in C\backslash K$ are arbitrary. Let $i_0,j_0,i'_0,j'_0$ such that $u\in \Gamma_{i_0j_0}(D_\epsilon\backslash\{0\})$ and $u'\in \Gamma_{i'_0j'_0}(D_\epsilon\backslash\{0\})$ and suppose that $1/\epsilon^{d_{i_0j_0}}\leq1/\epsilon^{d_{i'_0j'_0}}$. 
Let $R=1/\epsilon^{d_{i_0j_0}}$ and choose a path $\alpha:[0,1]\to \C\backslash D_R$ such that $\alpha(0)=P(u),\alpha(1)=P(u')$ and $\len(\alpha)\leq\pi R |P(u)-P(u')|$.
Consider the lifting $\tilde \alpha$ of $\alpha$ by $P|_{ \Gamma_{i_0j_0}(D_\epsilon\backslash\{0\})}$ with origin $u$ and let $u''$ be its end.  
We obviously have                                   
$$d(u,u')\leq  \len( \tilde\alpha) + d(u',u'')\,.$$   

On the other hand, since $P$ is linear, $dP_p=P|_{T_pC}$.
Thus 
\begin{equation*}
\frac{1}{\sqrt{1+c^2}}\leq||dP_p||\leq 1 \text{ for all } p\in C\backslash K.
\end{equation*}

In particular, $\len( \tilde\alpha) \leq \sqrt{1+c^2}\len(\alpha)\leq \pi R\sqrt{1+c^2}|P(u)-P(u')|$, as $|P(u)-P(u')| \leq d(u,u')$, we obtain
$$  \len( \tilde\alpha) \leq \pi R\sqrt{1+c^2} d(u,u').$$

If we join the segment $[u,u']$ to $ \tilde\alpha$ at $u$, we have a curve from $u'$ to $u''$, so   $d(u',u'')  \leq (1+\pi R\sqrt{1+c^2}) d(u,u')$.  
Finally, 
$$ \len( \tilde\alpha) + d(u',u'')  \leq (1+2\pi R\sqrt{1+c^2}) d(u,u'),$$
and the constant  $M=1+2  \pi R\sqrt{1+c^2} $ satisfies the desired conditions. 
\end{proof}

\begin{remark}\label{bilip}
In the above lemma, we prove that $P|_{C\backslash K}:C\backslash K\to \C\backslash P(K)$ is a covering map. Moreover, $P|_{C\backslash K}$ has derivative bounded above and below by positive constants. In particular, for a non-constant arc $\alpha$ the quotient $$\len(\tilde\alpha)/\len(\alpha )$$ is bounded above and below by positive constants.   

\end{remark}

The demonstration technique of $(\ref{it1})\Rightarrow (\ref{it2})$ the Theorem \ref{atinfinity} is  similar to  the  case of germ of complex curves in \cite{Anne}.
In particular, it is based on a so-called “bubble trick” argument.

\begin{proof}[Proof of $(\ref{it1})\Rightarrow (\ref{it2})$ of Theorem \ref{atinfinity}] 
We first prove that the Lipschitz geometry at infinity gives us the number of points at infinity.
Let $f\in \C[x,y]$ be a polynomial that defines $C$  which does not have  multiple factors. Let $n=\deg f$, then by a linear change of coordinates if necessary, we can assume that the monomial $y^n$ has  coefficient equal to 1 in $f$. 

The points at infinity of $C$ are the points $[x:y:0]\in\P^2$ satisfying $f_n(x,y)=0$, where $f_n$ denotes the homogeneous polynomial composed by the monomials in $f$ of degree $n$, so $[0:1:0]$ is not a point at infinity of $C$.

We claim that there are constant $c>0$ and an open Euclidean ball $B_{R_0}(0)$ of radius $R_0$ centered at the origin such that $|y|\leq c|x|$ for all $(x,y)\in C\backslash B_{R_0}(0)$.
Indeed, otherwise, there exists a sequence $\{z_k=(x_k,y_k)\}\subset C$ such that $\lim\limits_{k\to+\infty}\|z_k\|=+\infty$ and $|y_k|> k|x_k|$. 
Thus, taking a subsequence,  one can suppose that $\lim\limits_{k\to+\infty}\frac{y_k}{|y_k|}=y_0$ for some $y_0$ such that $|y_0|=1$.
Since $\frac{|x_k|}{|y_k|}< \frac{1}{k}$, $\lim\limits_{k\to+\infty}\frac{z_k}{\|z_k\|}=(0,y_0)$.
     On the other hand,
 \begin{eqnarray*}
0 = f(z_k)  =  f\left(\|z_k\| \frac{z_k}{\|z_k\|}\right) 
=
\|z_k\|^n \sum_{i=0}^{n}\frac{1}{\|z_k\|^{n-i}}f_{i} \left(\frac{z_k}{\|z_k\|}\right),
\end{eqnarray*}
 where $f_i$ denotes the homogeneous polynomial composed by the monomials in $f$ of degree $i$. This implies that
\begin{eqnarray*}
0 = f(z_k)= \sum_{i=0}^{n}\frac{1}{\|z_k\|^{n-i}}f_{i} \left(\frac{z_k}{\|z_k\|}\right),
\end{eqnarray*}
Letting $k \to \infty$ yields $f_n(0,y_0) = 0$, which implies that $[0:1:0]$ is a point at infinity of $C$, this is a contradiction. 
Therefore, the claim is true.

 Now, let  $[1:a_j:0], j=1,\ldots,m\leq n$ be the points at infinity of $C$. 
 We define cones
 $$V_j:=\{(x,y)\in\C^2:|y-a_jx|\leq\epsilon |x|\}$$
where $\epsilon>0$ is small enough that the cones are disjoint except at $0$. 
Then increasing $R_0>0$, if necessary, 

\[ C\backslash B_{R_0}(0)\subset  \bigcup_{j=1}^{m}V_j.\]

     Indeed, otherwise, there exists a sequence $\{z_k=(x_k,y_k)\}\subset C$ such that $\lim\limits_{k\to+\infty}\|z_k\|=+\infty$ and $|y_k-a_jx_k|> \epsilon|x_k|$ for all $j=1,\ldots,m$.
     Again, since $\|z_k\|\to+\infty$ as $k\to\infty$, we have
     $$\lim_{k\to\infty}f_n \left(\frac{z_k}{\|z_k\|}\right)=0.$$
     
     On the other hand, writing $f_n(x,y)=\prod_{j=1}^m(y-a_jx)^{d_j}$, where $d_j$ is  a positive integer such that $n=\sum_{1\leq j\leq m} d_j$, we have
$$\left\|f_n \left(\frac{z_k}{\|z_k\|}\right)\right\|=\frac{\prod_{j=1}^{m}|y_k-a_jx_k|^{d_j}}{\|z_k\|^n}\geq\left(\frac{\epsilon|x_k|}{\|z_k\|}\right)^n.$$
But, because of the first claim, we have 
$$\frac{|x_k|}{\|z_k\|}=\frac{1}{\sqrt{1+\left|\frac{y_k}{x_k}\right|^2}}\geq \frac{1}{\sqrt{1+c^2}},$$
which derives a contradiction.

We denote by $\mathcal{C}_j$ the part of $C\backslash B_{R_0}(0)$ inside $V_j$. Now, let $K,K'\subset \C^2$ be compact sets such that there is a bilipschitz map $\Phi:C\backslash K\to C'\backslash K'$. 
 Let  $[1:a'_j:0], j=1,\ldots,m'$ be the points at infinity of $C'$. We repeat the above arguments for $C'$, then increasing $R_0>0$, if necessary, 
$$C'\backslash B_{R_0}(0)\subset \bigcup_{j=1}^{m'}V'_j,\:\text{ where }\:V'_j:=\{(x,y)\in\C^2: |y-a'_jx|\leq\epsilon|x|\}.$$
Likewise, denote by $\mathcal{C}'_j$ the set $(C'\backslash B_{R_0}(0))\cap V'_j$.  
We have $\Phi(C\backslash B_R(0))\subset C'\backslash B_{h(R)}(0) $ with  $h(R)=\Theta(R) $. Since $\dist(\mathcal{C}_j\backslash B_R(0),\mathcal{C}_k\backslash B_R(0))=\Theta(R)$ we have
 $$\dist(\Phi(\mathcal{C}_j\backslash B_R(0)),\Phi(\mathcal{C}_k\backslash B_R(0)))=\Theta(R).$$ 

Notice that the sets $\mathcal{C}'_l,l=1,\ldots,m'$ have the following property:  the distance between any two connected components of $\mathcal{C}'_l$ outside a ball of radius $h(R)$ around $0$ is $o(R)$. 
Then, we cannot have $$\Phi(\mathcal{C}_j\backslash B_R(0))\subset \mathcal{C}_l'\backslash B_{h(R)}(0) \text{ and } \Phi(\mathcal{C}_k\backslash B_R(0))\subset \mathcal{C}_l'\backslash B_{h(R)}(0)$$
for $k\neq j$ then $m\leq m'$ and using the inverse $\Phi^{-1}$ we get $m=m'$.

Now, we obtain the topological type of $\widetilde{C}\cup L_\infty$  at the points at infinity. 
Without loss of generality, we can suppose that $[1:a_1:0]=[1:0:0]$ is a point at infinity for $C$.
We extract the characteristic and the coincidence exponents   relative to $L_\infty$ of the curve $(\widetilde{C}\cup L_{\infty},[1:0:0])$
 using the coordinate system and the induced Euclidean metric $d$ on $\mathcal{C}_1$.
Next, we prove that these data determine the  topology type of $(\widetilde{C}\cup L_{\infty},[1:0:0])$.
Finally, we prove that these data can be obtained without using the chosen coordinate system and even using the equivalent metric $d'$ induced by $\Phi$, for this we operate the ``bubble trick".

Let $U=\{[x:y:z]\in \P^2:x\neq0\}$ and consider the coordinate chart $\varphi:U\to\C^2$ defined by $\varphi([x:y:z])=(z/x,y/x)=(u,v)$.
In this local coordinates, $\varphi([1:0:0])$ is the origin and  we have $\ord_v (\widetilde{f}\circ \varphi^{-1})(0,v)=d_1$. 
 Let  $B_{1},\ldots,B_{k_1} $ be the branches of $(\varphi(\widetilde{C}\cap U),0)$.
Every branch of the curve $(\varphi(\widetilde{C}\cap U),0)$ has a Newton-Puiseux parametrization of the form 
$$\gamma_{s}(w)=\left(w^{d_{1s}},\sum_{k> 0} a_{sk} w^k\right),$$
where $d_{1s}$ are positive integers such that $\sum_{s=1}^{k_1} d_{1s}=d_1$.
Then, increasing  $R_0>0$ if necessary, the images of the   maps
$$\displaystyle\Gamma_{s}(w)=(\iota^{-1}\circ\varphi^{-1}\circ\gamma_s)(w)=\left(\frac{1}{w^{d_{1s}}},\frac{1}{w^{d_{1s}}}\sum_{k> 0} a_{sk} w^{k}\right), s=1,\ldots, k_1$$
cover $\mathcal{C}_1$. Therefore, the lines $x=t$ for $t\in (R_0,\infty)$ intersect $\mathcal{C}_1$ in $d_1$ points  
  $p_1(t),\dots,p_{d_1}(t)$ which depend continuously on $t$.
  Denote by $[d_1]$ the set $\{1,\ldots,d_1\}$.
  For each $  j, k \in [d_1]$ with $j < k$, the distance $d(p_j(t),p_k(t))$ has the form $\Theta(t^{1-q(j,k)})$, where $q(j,k) = q(k,j)$ is either a characteristic Puiseux
exponent relative to $L_\infty$ for a branch of the plane curve $(\widetilde{C}\cup L_\infty,[1:0:0])$  or a coincidence exponent relative to $L_\infty$ between two branches  of $(\widetilde{C}\cup L_\infty,[1:0:0])$.
 For $j\in[d_1]$ define $q(j,j)=\infty$.  
 
\begin{lemma} \label{le:curve geometry}The map $q\colon
  [d_1]\times[d_1]\to \Q\cup\{\infty\}$, $(j,k)\mapsto q(j,k)$, determines the  topological type of $(\widetilde{C}\cup L_\infty,[1:0:0])$.
\end{lemma}

\begin{proof} 
The topological type of $(\widetilde{C}\cup L_\infty,[1:0:0])$ is encoded by its Eggers-Wall tree relative to a smooth branch $\mathcal{L}$ transversal to $(\widetilde{C}\cup L_\infty,[1:0:0])$ (see Wall \cite[Proposition 4.3.9 and Theorem 5.5.9]{Wall}).   
To prove the lemma we notice that the function $q$ is the same as the function $\ord_x$ of   Definition \ref{carouseultree}. By the process described in Section \ref{trees}, one obtains the Eggers-Wall tree relative to $L_\infty$ of $(\widetilde{C}\cup L_\infty,[1:0:0])$. By applying the inversion theorem for Eggers-Wall tree \cite[Theorem 4.5]{GarciaBarroso2019} to $\Theta_{L_{\infty}}(\widetilde{C}\cup L_\infty\cup \mathcal{L},[1:0:0])$, one  gets the Eggers-Wall tree $\Theta_{\mathcal{L}}(\widetilde{C}\cup L_\infty,[1:0:0])$.

\end{proof}

As already noted, this discovery of the  topology type involved the  chosen coordinate system and the metric $d$.
We must show we can discover it using  $d'$ and without using  the  chosen coordinate system.

The points $p_1(t),\dots,p_{d_1}(t)$ that we used to find the numbers $q(j,k)$ were obtained by intersecting $\mathcal{C}_1$ with the line $x=t$. 
The arc  $t\in (R_0,\infty)\mapsto p_1(t)$ satisfies
\begin{equation}\label{1}
d(0,p_1(t))=\Theta(t).    
\end{equation}

Moreover, the other points $p_2(t),\dots,p_{d_1}(t)$ are in the disk of radius $\eta t$ centered at $p_1(t)$ in the plane $x=t$. 
Here, $\eta>0$ can be as small as we like, so long as $R_0$ is then chosen sufficiently big.

Instead of a disk of radius $\eta t$, we can use a ball $B(p_1(t),\eta t)$ of radius $\eta t$ centered at $p_1(t)$. 
This ball $B(p_1(t),\eta t)$ intersects $\mathcal{C}_{1}$ in $d_1$ topological disks $D_1(\eta t),\dots,D_{d_1}(\eta t)$, named such that $p_l(t)\in D_l(\eta t),l=1,\ldots,d_1$ and thus  $\dist(D_j(\eta t),D_k(\eta t))\leq d(p_j(t),p_k(t))$.
On the other hand, let $\widetilde{p}_l(t)\in D_l(\eta t), l=1,\ldots,d_1$ such that $$\dist(D_j(\eta t),D_k(\eta t))= d(\widetilde{p}_j(t),\widetilde{p}_k(t)).$$ 

Consider the projection $P\colon \C^2 \to \C$ given by $P(x,y)=x$ and let $\alpha_t $ be the segment in $\C$ joining $P(\widetilde{p}_j(t))$ to $P(\widetilde{p}_k(t))$ and let $\tilde{\alpha}_t$ be the lifting of $\alpha_t$ by the restriction
$P|_{C\backslash B_{R_0}(0)}$  with origin $ \widetilde{p}_k(t).$ 
Applying Lemma \ref{semnome} to  $P$ with
$u=\widetilde{p}_k(t)$ and $u'=\widetilde{p}_j(t)$, we then obtain 
$$d(\widetilde{p}_j(t),\widetilde{p}_k(t))\geq \frac{1}{M}(\len(\tilde{\alpha}_t)+d(\widetilde{p}_j(t),\tilde{\alpha}_t(1)))\geq \frac{1}{M}d(\widetilde{p}_j(t),\tilde{\alpha}_t(1)).$$ 
But $d(\widetilde{p}_j(t),\tilde{\alpha}_t(1))=\Theta(t^{1-q(j,k)})$ since $P(\widetilde{p}_j(t))=P(\tilde{\alpha}_t(1))$ and $|P(\widetilde{p}_j(t))|=\Theta(t)$. 

We now replace the arc $p_1$ by any continuous arc  on $\mathcal{C}_1$ satisfying (\ref{1}) and we still denote this new arc by $p_1$.
If $\eta$ is sufficiently small it is still true that $B_{\mathcal{C}_1}(p_1(t),\eta t):=\mathcal{C}_1\cap B(p_1(t),\eta t)$ consists of $d_1$ disks $D_1(\eta t),\dots,D_{d_1}(\eta t)$ with $\dist\bigl(D_j(\eta t),D_k(\eta t)\bigr)=\Theta(t^{1-q(j,k)})$.
So at this point, we have gotten rid of the dependence on the choice of coordinate system in discovering the topology, but not yet dependence on the metric $d$. 

A $L$-bilipschitz change to the metric may make the components of $B_{\mathcal{C}_1}(p_1(t),\eta t)$ disintegrate into many pieces, so we can no longer simply use distance between all pieces. 
To resolve this, we consider  $B_{\mathcal{C}_1}(p_1(t),\eta t/L)$ and $B_{\mathcal{C}_1}(p_1(t),\eta L t)$. 
Note that 
$$B_{\mathcal{C}_1}(p_1(t),\eta t/L)\subset B'_{\mathcal{C}_1}(p_1(t),\eta t)\\ \subset B_{\mathcal{C}_1}(p_1(t),\eta Lt),
$$
 where $B'$ means we are using the modified metric $d'$.

Denote by $D_j(\eta t/L)$ and $D_j(\eta Lt),j=1,\dots, d_1$ the disk of $B_{\mathcal{C}_1}(p_1(r),\eta t/L)$ and  $B_{\mathcal{C}_1}(p_1(r),\eta Lt)$, respectively, so that  $D_j(\eta t/L)\subset D_j(\eta Lt)$ for $j=1,\ldots,d_1$. Thus  $B'_{\mathcal{C}_1}(p_1(t),\eta t)$ has  $d_1$ components such that each one contains at most one component of $B_{\mathcal{C}_1}(p_1(r),\eta t/L)$.
Therefore, exactly $d_1$ components of $B'_{\mathcal{C}_1}(p_1(t),\eta t)$  intersect $B_{\mathcal{C}_1}(p_1(t), \eta t/L)$. 
Naming these components $D'_1(\eta t),\dots,D'_{d_1}(\eta t)$, such that $D_j(\eta t/L)\subset D'_j(\eta t) \subset D_j(\eta Lt),j=1,\ldots,d_1$, we still have $\dist(D'_j(\eta t),D'_k(\eta t))=\Theta(t^{1-q(j,k)})$ since 
$$
 \dist(D_j(\eta L t),D_k(\eta Lt))\leq \dist(D'_j(\eta t),D'_k(\eta t))\\ \leq \dist(D_j(\eta t/L),D_k(\eta t/L)).
$$
So the $q(j,k)$ are determined by the distance between  $D_j'(\eta t), j=1,\ldots,d_1$.

Up to now, we have used the metric $d$ to select the components $D_j'(\eta t), j=1,\ldots,d_1$ of $B'_{\mathcal{C}_1}(p_1(t),\eta t)$. 
To avoid using the metric $d$,  consider  $B'_{\mathcal{C}_1}(p_1(t), \eta t/L^2)$. We have  
 $$
  B_{\mathcal{C}_1}(p_1(t), \eta t/L^3) \subset B_{\mathcal{C}_1}'(p_1(t), \eta t/L^2) \subset B_{\mathcal{C}_1}(p_1(t), \eta t/L)\subset D_1'(\eta t)\cup\dots\cup D_{d_1}'(\eta t).
$$ 
This implies that $B'_{\mathcal{C}_1}(p_1(t), \eta t/L^2)$ intersects only the components $D'_j(\eta t), j=1,\dots,d_1 $ of $B'_{\mathcal{C}_1}(p_1(t),\eta t)$.
So we can use only the metric $d'$ to select these components and we are done.

\end{proof}

\section{Topological type determines Lipschitz geometry at infinity}\label{topsimplies}
In this section, we prove that (\ref{it2}) implies (\ref{it1})  of Theorem \ref{atinfinity}. For this, we will construct a bilipschitz map between complex algebraic plane curves with the same data in (\ref{it2}).

\begin{proof}[Proof of the  implication $(\ref{it2}) \Rightarrow  (\ref{it1})$ of Theorem \ref{atinfinity}]
Let $C_1$ and $C_2$ be complex algebraic plane curves with the same data described by (\ref{it2}) of of Theorem \ref{atinfinity}.
Choose $(x,y)$ coordinates in such a  way that none of the curves have  the point $[0:1:0]$ as a point at infinity. 

Let $[1:a^l_1:0],\ldots,[1:a^l_{m_l}:0]$ be the points at infinity of $C_l, l=1,2$, denoted in such a way that $(\widetilde{C}_1\cup L_\infty,[1:a^1_i:0])$ has the same topological type as $(\widetilde{C}_2\cup L_\infty,[1:a^2_i:0])$. Then, by \cite[Theorem 5.5.9]{Wall} and \cite[Proposition 4.3.9]{Wall}, for any smooth branch $L_1$ (resp. $L_2$) through $[1:a_i^1:0]$ (resp. $[1:a_i^2:0]$) transversal to $(C_1\cup L_{\infty},[1:a_i^1:0])$ (resp. $(C_1\cup L_{\infty},[1:a_i^2:0])$) the Eggers-Wall trees $\Theta_{L_1}(\widetilde{C}_1\cup L_{\infty},[1:a_i^1:0])$ and $\Theta_{L_2}(\widetilde{C}_2\cup L_{\infty},[1:a_i^2:0])$ are isomorphic. 
Then, we apply the inversion theorem for Eggers-Wall tree \cite[Theorem 4.5]{GarciaBarroso2019} to both and we get that $\Theta_{L_\infty}(\widetilde{C}_1,[1:a_i^1:0])$ and $\Theta_{L_\infty}(\widetilde{C}_2 ,[1:a_i^2:0])$ are isomorphic.

For each $i$, let  $B^{l}_{i1},\ldots,B^{l}_{ik_i} $ be the branches of $(\widetilde{C}_l,[1:a^l_i:0]), l=1,2$. Again, we denoted in such a way that $(B_{ij}^1,[1:a^1_i:0])$ has the same topological type as $(B_{ij}^2,[1:a^2_i:0])$. From what has been said above, we have that $B_{ij}^1$ and $B_{ij}^2$ have the same characteristic exponents relative to $L_\infty$ and $k_{L_\infty}(B_{ij}^1,B_{ij'}^1)=k_{L_\infty}(B_{ij}^2,B_{ij'}^2).$

The open set $U=\{[x:y:z]\in \P^2:x\neq0\}$ contains all the points at infinity of $C_l,l=1,2$. 
We can use   the coordinate chart $\varphi:U\to\C^2$  defined by $\varphi([x:y:z])=(z/x,y/x)$ to obtain a Newton-Puiseux parametrization of the branches $\varphi(B^l_{ij})$.
Let $D_{\epsilon_0}$ be the open disk of radius $\epsilon_0>0$ centered at the origin with $\epsilon_0$ sufficiently small such that there  exist Newton-Puiseux parametrization $\gamma^l_{ij}:D_{\epsilon_0}\to\C^2$ of $\varphi(B^l_{ij})$ given by
$$\gamma^l_{ij}(w)=\left(w^{d_{ij}},a_i^l+\sum_{k>0}a^l_{ijk}w^k\right).$$  
Let $\Gamma^{l}_{ij}:D_{\epsilon_0}\backslash\{0\}\to\C^2$ given by  
 $$\Gamma_{ij}^l(w)=(\iota^{-1}\circ\varphi^{-1}\circ\gamma^l_{ij})(w)=\left(\frac{1}{w^{d_{ij}}},\frac{a_i^l+\sum_{k>0}a^l_{ijk}w^k}{w^{d_{ij}}}\right),l=1,2.$$

 Consider the compact set $K^l_\epsilon=C_{l}\backslash \bigcup_{ij} \Gamma^l_{ij}(D_\epsilon\backslash\{0\}),l=1,2$.
 We will prove that there exists $\epsilon>0$ that the map 
 
 $$\begin{array}{rcl}
\Phi:C_1\backslash K^1_\epsilon&\longrightarrow&C_2\backslash K^2_\epsilon\\
\Gamma^1_{ij}(w)&\longmapsto&\Gamma^2_{ij}(w)
\end{array}$$
is bilipschitz.
\begin{claim}
Consider the projection $P\colon \C^2 \to \C$ given by $P(x,y)=x$.
In order to check that $\Phi$ is a Lipschitz map it is enough to consider  points in $C_1\backslash K^1_\epsilon$ with the same $x$ coordinate. That is, there exists a constant $c>0$ such that   
 $$d\bigl(\Gamma^2_{ij}(w'),\Gamma^2_{i'j'}(w'')\bigr)\leq c d\bigl(\Gamma^1_{ij}(w'),\Gamma^1_{i'j'}(w'')\bigr),$$
  for all $w',w''$ such that $P(\Gamma_{ij}^1(w'))=P(\Gamma_{i'j'}^1(w''))$.
\end{claim}

Indeed, let $\Gamma_{ij}^1(w)$ and $\Gamma_{i'j'}^1(w')$ be any two elements of $C_1\backslash K^1_\epsilon$ and  suppose that $1/\epsilon^{d_{ij}}\leq 1/\epsilon^{d_{i'j'}}$.
Let $\alpha$ be a curve in $\C\backslash D_{1/\epsilon^{d_{ij}}}$ joining $P(\Gamma_{ij}^1(w))$ to $P(\Gamma_{i'j'}^1(w'))$ as in  Lemma \ref{semnome}. Let $\tilde{\alpha}_1$ (resp.\ $\tilde{\alpha}_2$) be the lifting of $\alpha$ by the restriction 
$P|_{\Gamma_{ij}^1(D_{\epsilon}\backslash \{0\})}$ (resp.\ $P|_{\Gamma_{ij}^2(D_{\epsilon}\backslash \{0\})}$) with origin $ \Gamma_{ij}^1(w)$ (resp.\ 
$\Gamma_{ij}^2(w)$). 
Consider the unique $w'' \in D_{\epsilon}$ such that $\Gamma^1_{ij}(w'')$ is the end of $\tilde{\alpha}_1$.  Notice that   $P\circ\Gamma^1_{ij}=P\circ\Gamma^2_{ij}$ and  by uniqueness of lifts $\tilde{\alpha}_2=\Gamma^2_{ij}\circ(\Gamma^1_{ij})^{-1}\circ \tilde{\alpha}_1 $ which implies that $\Gamma^2_{ij}(w'')$ is the end of $\tilde{\alpha}_2$.

We have
$$
d\bigl(\Gamma^2_{ij}(w),\Gamma^2_{i'j'}(w')\bigr) \leq \len(\tilde \alpha_2) +
d\bigl(\Gamma^2_{ij}(w''),\Gamma^2_{ij}(w')\bigr).$$

According to the Remark \ref{bilip}, there are constant, say $c_1$ and $c_2$ such that 
 $ \len(\tilde \alpha_2) \leq c_1\len(\alpha)\leq c_1c_2 \len(\tilde \alpha_1)$. By hypothesis, there exists a constant $c>0$ such that   
 $$d\bigl(\Gamma^2_{ij}(w''),\Gamma^2_{ij}(w')\bigr)\leq c d\bigl(\Gamma^1_{ij}(w''),\Gamma^1_{ij}(w')\bigr).$$
Therefore setting $C=\max\{c_1c_2,c\}$, we obtain
$$d\bigl(\Gamma^2_{ij}(w),\Gamma^2_{i'j'}(w')\bigr) \leq C\Bigl( \len(\tilde \alpha_1) +  
d\bigl(\Gamma^1_{ij}(w''),\Gamma^1_{ij}(w')\bigr)\Bigr).$$
Applying Lemma \ref{semnome} to  $C_1$ with
$u=\Gamma^1_{ij}(w)$ and $u'=\Gamma^1_{i'j'}(w')$, we  then have
$$
d\bigl(\Gamma^2_{ij}(w),\Gamma^2_{i'j'}(w')\bigr)  \leq C M
d\bigl(\Gamma^1_{ij}(w),\Gamma^1_{i'j'}(w')\bigr) .$$
This proves $\Phi$ is Lipschitz 
and the claim.

Now, let $B^1_{ij}$ and $B^2_{i'j'}$ be branches of $\widetilde{C}_1$ and $\widetilde{C}_2$, respectively,  
 with $i\neq i'$. Let  $s\in(0,1]\to \Gamma^1_{ij}(ws^{1/d_{ij}})$ and $s\in(0,1]\to \Gamma^1_{i'j'}(w's^{1/d_{i'j'}})$  be the two real arcs with $w^{d_{ij}}=(w')^{d_{i'j'}}$. Then we have
\begin{multline*}
    d\bigl(\Gamma^1_{ij}(ws^{1/d_{ij}}), \Gamma^1_{i'j'}(w's^{1/d_{i'j'}})\bigr) = \frac{1}{s|w^{d_{ij}}|}   
\bigg| a^1_{ij}-a^1_{i'j'}
+ \sum_{k>0}a^1_{ijk}w^k{s^{k/d_{ij}}}  \\ -\sum_{k>0}a^1_{i'j'k}(w')^k{s^{k/d_{i'j'}}} \bigg|.
\end{multline*}
and
\begin{multline*}
    d\bigl(\Phi(\Gamma^1_{ij}(ws^{1/d_{ij}})), \Phi(\Gamma^1_{i'j'}(w's^{1/d_{i'j'}})\bigr) = \frac{1}{s|w^{d_{ij}}|}   
\bigg| a^2_{ij}-a^2_{i'j'}
+ \sum_{k>0}a^2_{ijk}w^k{s^{k/d_{ij}}}  \\ -\sum_{k>0}a^2_{i'j'k}(w')^k{s^{k/d_{i'j'}}} \bigg|
\end{multline*}
Hence the ratio
\begin{equation}\label{quotient1}
 d\bigl(\Gamma^1_{ij}(ws^{1/d_{ij}}), \Gamma^1_{i'j'}(w's^{1/d_{i'j'}})\bigr)\, \Big/ \, 
d\bigl(\Phi(\Gamma^1_{ij}(ws^{1/d_{ij}})), \Phi(\Gamma^1_{i'j'}(w's^{1/d_{i'j'}})\bigr)\Bigr)    
\end{equation}
 tends to the non-zero constant $ \frac{|a^1_{ij}-a^1_{i'j'}|}{|a^2_{ij}-a^2_{i'j'}|}$ as $s$ tends to $0$ for every such pairs $(w,w')$. So there
exists $\epsilon >0$ such that for each such $(w,w')$ with $|w| = 1$ and
each $s < \epsilon$, the quotient (\ref{quotient1}) belongs to $[1/c,c]$ where
$c>0$.

Now, consider the branches   $B^1_{ij}$ and $B^2_{ij}$. 
Let  $s\in(0,1]\to \Gamma^1_{ij}(ws)$ and $s\in(0,1]\to \Gamma^1_{i'j'}(w's)$ be the two real arcs with $w^{d_{ij}}=(w')^{d_{ij}}$.
Then we have
$$    d\bigl(\Gamma^1_{ij}(ws), \Gamma^1_{ij}(w's)\bigr) = \frac{1}{s^{d_{ij}}|w^{d_{ij}}|}   
\bigg| 
\sum_{k>0}a^1_{ijk}(w^k-(w')^k){s^{k}}   \bigg|
$$
and
$$
    d\bigl(\Phi(\Gamma^1_{ij}(ws)), \Phi(\Gamma^1_{ij}(w's)\bigr)  = \frac{1}{s^{d_{ij}}|w^{d_{ij}}|}   
\bigg| 
\sum_{k>0}a^2_{ijk}(w^k-(w')^k){s^{k}}   \bigg|
$$

Let $k_0$ be the minimal element of $\{k ;a^1_{ijk}\neq 0 \text{ and } w^k \neq(w')^k \}$. 
Then $k_0/d_{ij}$ is a characteristic exponent for $B^1_{ij}$ relative to $L_\infty$, so $a^1_{ijk_0}$ and $a^2_{ijk_0}$ are non-zero.
Hence the ratio
\begin{equation}\label{quotient2}
 d\bigl(\Gamma^1_{ij}(ws), \Gamma^1_{ij}(w's)\bigr)\, \Big/ \, 
d\bigl(\Phi(\Gamma^1_{ij}(ws)), \Phi(\Gamma^1_{ij}(w's)\bigr)\Bigr)  
\end{equation}
tends to the non-zero constant $c_{ijk_0}= \frac{|a^1_{ijk_0}|}{|a^2_{ijk_0}|}$ as $s$ tends to $0$.

Notice that the integer $k_0$ depends on the pair of points $(w,w')$.
But $k_0/d_{ij}$ is a characteristic exponent relative to $L_\infty$ of $B^1_{ij}$.
Therefore there is a finite number of values for $k_0$ and $c_{ijk_0}$.  
Moreover, the set of pairs $(w,w')$ such that $w\neq w'$ and $w^{d_{ij}}=(w')^{d_{ij}}$  consists of a disjoint union of $d_{ij}-1$ lines, say $L_l=\{(w,\exp(2\pi l/d_{ij})w),w\in \C^*\}, l=1,\ldots, d_{ij}-1$.
Observe that for any $(w,w')\in L_l$ the quotient (\ref{quotient2}) tends to positive constant as $s\to0$ which does not depend  on the pair $(w,w')$.
So there exists $\epsilon_1 >0$ such that for each such $(w,w')$ with $|w| = 1$ and each $s \leq \epsilon_1$, the quotient (\ref{quotient2}) belongs to $[1/c,c]$ where
$c>0$, as claimed.

For the case of branches $B_{ij}^1$ and $B_{ij'}^2$ with $j\neq j'$, the same arguments work taking into account their coincidence exponent relative to $L_\infty$.

\end{proof}

\section{Lipschitz geometry of complex algebraic plane curves }

In this section, we present the complete  classification of the Lipschitz geometry of complex algebraic plane curves. We define the Lipschitz graph of a complex algebraic plane curves which is a combinatorial object that encode its Lipschitz geometry.

Let $C$ be a complex algebraic plane curve.
A sequence of \textbf{good minimal resolution} of $\widetilde{C}$ produces a smooth curve  $\widetilde{\mathcal{C}}$. By \cite[Lemma 9.2.3]{brieskorn}, the connected components of  $\widetilde{\mathcal{C}}$ correspond bijectively to the irreducible components of $C$. 

\begin{definition}\label{Lipgraph}
Let $C$ be a complex algebraic plane curve with irreducible components $C_1,\ldots, C_n$.
A sequence of \textbf{good minimal resolution} of $\widetilde{C}\cup L_\infty$    produces a sequence of exceptional curves $E_1,\ldots,E_{r}$ and strict transform curves $\mathcal{C}_1,\ldots, \mathcal{C}_{n}$ of the curves  $C_1,\ldots, C_{n}$ and the strict transform $\mathcal{L}_\infty$ of the line at infinity $L_\infty$.

The \textbf{Lipschitz graph of $C$}  is a rooted graph with vertices $V_k$ corresponding to the curves $E_k$ labeled with its  self-intersection number,  vertices $W_i$ corresponding to the  curves $\mathcal{C}_i$ not labeled and a root corresponding to the $\mathcal{L}_\infty$. 
We put one edge joining vertices for each intersection point of  the corresponding curves.    
\end{definition}

\begin{remark}\label{disunion}
Let $p_1,\ldots,p_m$ be the singular points of $\widetilde{C}\cup L_\infty$. 
We point out that the Lipschitz graph  of $C$ is obtained as the quotient of the disjoint union of the individual dual resolution graph of minimal good resolutions of $(\widetilde{C}\cup L_\infty,p_i)$, by identifying
all vertices corresponding to the branch of an irreducible component $C_j$ for all $j=1,\ldots,n$.
Then it is clear that the Lipschitz graph of $C$ is determined by the topological type of the germs  $(\widetilde{C}\cup L_\infty,p_i)$ and $(\widetilde{C}_j\cup L_\infty,p_i)$.
\end{remark}

\begin{example}\label{quotient}
Let $C$ be a complex algebraic plane curve defined by $x[y^2-x^2(1+x)]=0$. We have two singular points for $\widetilde{C}\cup L_\infty$, namely $[0:0:1]$ and $[0:1:0]$. 
Dual resolution graph of a good minimal resolution at the singular point $[0:0:1]$ is drawn in Figure \ref{fig:dualresolution}.

\begin{figure}[h]
    \centering    
\begin{tikzpicture}[line cap=round,line join=round,>=triangle 45,x=1cm,y=1cm]
\clip(-2,-1.8) rectangle (8,1);
\draw [color=blue, line width=1pt ] plot [smooth, tension=2] coordinates { (1,1) (0,0)  (-1,0) (0,0) (1,-1)};
\draw [line width=1pt, color=green] (0,-1)-- (0,1);

\begin{scope}[xshift=3cm]
\draw [line width=1pt,color=black] (-1,0) -- (2,0);
\draw [line width=1pt,color=green] (0,-1) -- (0,1);
\draw [color=blue, line width=1pt ] plot [smooth, tension=1] coordinates { (0.5,1) (1,-1)  (1.5,1)};
\end{scope}

\begin{scope}[xshift=7cm,yshift=-1cm]
\draw [->,line width=1pt,color=green] (0,0) -- (-1,2);
\draw [->,line width=1pt,color=blue] (0,0) -- (0,2);
\draw [->,line width=1pt,color=blue] (0,0) -- (1,2);
\begin{scriptsize}
\draw [fill=black] (0,0) circle (2pt);
\draw[color=black] (0,-0.5) node {\normalsize$-1$};
\end{scriptsize}
\end{scope}

\end{tikzpicture}

\caption{The blue arrow vertices correspond to the branches of $C_1: y^2-x^2(1+x)=0$ and the green correspond to the branch of $C_2: x=0$.}
    \label{fig:dualresolution}
\end{figure}
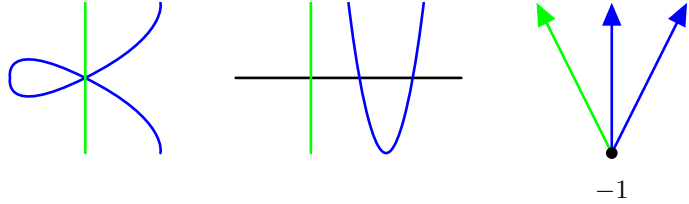

Dual resolution graph of a good minimal resolution for the singular point $[0:1:0]$ is drawn in Figure \ref{fig:dual}.

\begin{figure}[h]
    \centering
\begin{tikzpicture}[line cap=round,line join=round,>=triangle 45,x=0.8cm,y=0.8cm]
\clip(-1,-1) rectangle (5,2);
\draw [line width=1pt] (0,0)-- (2,0);
\draw [line width=1pt] (2,0)-- (4,0);
\draw [->,line width=1pt,color=green] (0,0) -- (0,2);
\draw [->,line width=1pt,color=blue] (4,0) -- (3,2);
\draw [->,line width=1pt,color=red] (4,0) -- (5,2);

\begin{scriptsize}
\draw [fill=black] (0,0) circle (2pt);
\draw[color=black] (0,-0.4) node {\normalsize$-2$};
\draw [fill=black] (2,0) circle (2pt);
\draw[color=black] (2,-0.4) node {\normalsize$-2$};
\draw [fill=black] (4,0) circle (2pt);
\draw[color=black] (4,-0.4) node {\normalsize$-1$};
\end{scriptsize}
\end{tikzpicture}
\caption{The blue arrow vertex correspond to the branch of $C_1: y^2-x^2(1+x)=0$ and the green correspond to the branch of $C_2: x=0$. 
The red arrow vertex corresponds to the branch of the line at infinity.}
    \label{fig:dual}
\end{figure}
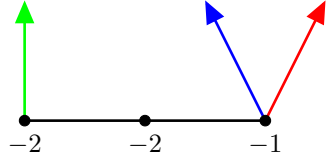
Connecting these graphs in the way described above we obtain the Lipschitz graph for $C$ (see Fig. \ref{fig:lipgraph}). 
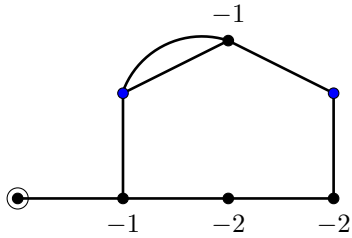
\begin{figure}[h]
    \centering
\begin{tikzpicture}[line cap=round,line join=round,>=triangle 45,x=0.7cm,y=0.7cm]

\draw [line width=1pt] (0,0)-- (2,0);
\draw [line width=1pt] (2,0)-- (4,0);
\draw [line width=1pt] (4,0)-- (6,0);
\draw [line width=1pt] (6,0)-- (6,2);
\draw [line width=1pt] (2,2)-- (2,0);
\draw [line width=1pt] (4,3)-- (2,2);
\draw [line width=1pt] (4,3)-- (6,2);
\draw [shift={(3.5,1.48)},line width=1pt]  plot[domain=1.2529983209850046:2.807890477052243,variable=\t]({1*1.6001249951175691*cos(\t r)+0*1.6001249951175691*sin(\t r)},{0*1.6001249951175691*cos(\t r)+1*1.6001249951175691*sin(\t r)});
\begin{scriptsize}

\draw [fill=black] (0,0) circle (2pt);
\draw [color=black] (0,0) circle (4pt);

\draw [fill=black] (2,0) circle (2pt);
\draw[color=black,font=\fontsize{15}{15}] (2,-0.5) node {\normalsize $-1$};
\draw [fill=black] (4,0) circle (2pt);
\draw[color=black,font=\fontsize{40}{40}] (4,-0.5) node {\normalsize $-2$};
\draw [fill=black] (6,0) circle (2pt);
\draw[color=black] (6,-0.5) node {\normalsize$-2$};
\draw [fill=blue] (6,2) circle (2pt);
\draw [fill=blue] (2,2) circle (2pt);
\draw [fill=black] (4,3) circle (2pt);
\draw[color=black] (4,3.5) node {\normalsize$-1$};
\end{scriptsize}
\end{tikzpicture}
    \caption{Lipschitz graph of Example \ref{quotient}. The vertices that correspond to irreducible components of the curve are distinguished from the other vertices by the fact that they are not labeled. 
But to improve the visualization of the graph we put a distinct color to such vertices.   
 }
    \label{fig:lipgraph}
\end{figure}
\end{example}

We can do the inverse process: start from a Lipschitz graph  of a complex algebraic plane curve $C$ and obtain the individuals  dual resolution graph of minimal good resolutions of $(\widetilde{C}\cup L_\infty,p_i)$. Then by \cite[Theorem 8.1.7]{Wall} we extract the following data: the topological type of the germ of the curve $\widetilde{C}\cup L_\infty$  at each of its singular points. 

\begin{example}
Suppose that Figure \ref{fig:given} is a Lipschitz graph of a complex algebraic plane curve $C$.
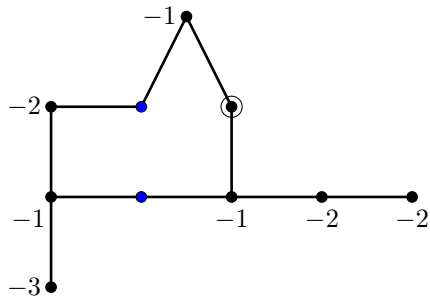
\begin{figure}[h]
    \centering
\begin{tikzpicture}[line cap=round,line join=round,>=triangle 45,x=0.6cm,y=0.6cm]
\draw [line width=1pt] (-3,0)-- (-1,0);
\draw [line width=1pt] (-1,0)-- (1,0);
\draw [line width=1pt] (1,0)-- (1,2);
\draw [line width=1pt] (1,0)-- (3,0);
\draw [line width=1pt] (3,0)-- (5,0);
\draw [line width=1pt] (-3,0)-- (-3,-2);
\draw [line width=1pt] (-3,0)-- (-3,2);
\draw [line width=1pt] (-3,2)-- (-1,2);
\draw [line width=1pt] (-1,2)-- (0,4);
\draw [line width=1pt] (0,4)-- (1,2);

\begin{scriptsize}
\draw [fill=black] (-3,0) circle (2pt);
\draw[color=black] (-3.5,-0.5) node {\normalsize	$-1$};

\draw [fill=blue] (-1,0) circle (2pt);

\draw [fill=black] (1,0) circle (2pt);
\draw[color=black] (1,-0.5) node {\normalsize	$-1$};

\draw [fill=black] (1,2) circle (2pt);
\draw [color=black] (1,2) circle (4pt);

\draw [fill=black] (3,0) circle (2pt);
\draw[color=black] (3,-0.5) node {\normalsize $-2$};

\draw [fill=black] (5,0) circle (2pt);
\draw[color=black] (5,-0.5) node {\normalsize$-2$};

\draw [fill=black] (-3,-2) circle (2pt);
\draw[color=black] (-3.6,-2) node {\normalsize$-3$};
\draw [fill=black] (-3,2) circle (2pt);
\draw[color=black] (-3.6,2) node {\normalsize$-2$};
\draw [fill=blue] (-1,2) circle (2pt);
\draw [fill=black] (0,4) circle (2pt);
\draw[color=black] (-0.6,4) node {\normalsize$-1$};
\end{scriptsize}
\end{tikzpicture}
\caption{A given Lipschitz graph of a complex algebraic plane curve $C$ is given.}
    \label{fig:given}
\end{figure}
 If we erase the vertices corresponding to the irreducible components of $\widetilde{C}\cup L_\infty$, we get three graphs with some no end edges.
We put an arrow vertex in each no end edges. 
But before doing all that we distinguish by colors the vertices corresponding to  irreducible components and the edges connected to them to discern which branches belongs to an irreducible component (see Fig. \ref{fig:colorvertex}). 
\begin{figure}[h]
    \centering
\begin{tikzpicture}[line cap=round,line join=round,>=triangle 45,x=0.6cm,y=0.6cm]
\draw [line width=1pt,color=blue] (-3,0)-- (-1,0);
\draw [line width=1pt,color=blue] (-1,0)-- (1,0);
\draw [line width=1pt,color=red] (1,0)-- (1,2);
\draw [line width=1pt] (1,0)-- (3,0);
\draw [line width=1pt] (3,0)-- (5,0);
\draw [line width=1pt] (-3,0)-- (-3,-2);
\draw [line width=1pt] (-3,0)-- (-3,2);
\draw [line width=1pt,color=green] (-3,2)-- (-1,2);
\draw [line width=1pt, color=green] (-1,2)-- (0,4);
\draw [line width=1pt,color=red] (0,4)-- (1,2);

\begin{scriptsize}
\draw [fill=black] (-3,0) circle (2pt);
\draw[color=black] (-3.5,-0.5) node {\normalsize	$-1$};

\draw [fill=blue] (-1,0) circle (2pt);

\draw [fill=black] (1,0) circle (2pt);
\draw[color=black] (1,-0.5) node {\normalsize	$-1$};

\draw [fill=green] (-1,2) circle (2pt);
\draw [fill=black] (3,0) circle (2pt);
\draw[color=black] (3,-0.5) node {\normalsize $-2$};

\draw [fill=black] (5,0) circle (2pt);
\draw[color=black] (5,-0.5) node {\normalsize$-2$};

\draw [fill=black] (-3,-2) circle (2pt);
\draw[color=black] (-3.6,-2) node {\normalsize$-3$};

\draw [fill=black] (-3,2) circle (2pt);
\draw[color=black] (-3.6,2) node {\normalsize$-2$};

\draw [fill=red] (1,2) circle (2pt);
\draw [color=red] (1,2) circle (4pt);

\draw [fill=black] (0,4) circle (2pt);
\draw[color=black] (-0.6,4) node {\normalsize$-1$};
\end{scriptsize}
\end{tikzpicture}
\caption{There are three irreducible components, say, $R$ (red) corresponding to the line at infinity, $G$ (green) and $B$ (blue) corresponding to the irreducible components of $C$.}
    \label{fig:colorvertex}
\end{figure}
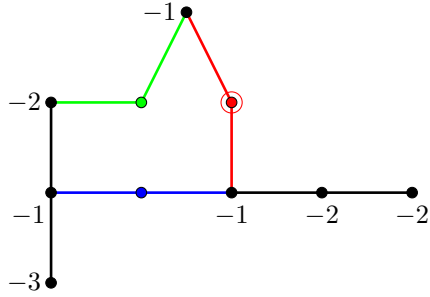

Now, we delete the vertex corresponding to irreducible components and put the arrows vertices in the no end edges  with the same color as the edge (see Fig.\ref{fig:cut}).

\begin{figure}[h]
    \centering
\begin{tikzpicture}[line cap=round,line join=round,>=triangle 45,x=0.6cm,y=0.6cm]

\draw [->,line width=1pt,color=blue] (-3,0)-- (-1.2,0);
\draw [line width=1pt] (-3,0)-- (-3,-2);
\draw [line width=1pt] (-3,0)-- (-3,2);
\draw [->,line width=1pt, color=green] (-3,2)-- (-1.2,2);

\draw [<-,line width=1pt, color=blue] (-0.8,0)-- (1,0);
\draw [->,line width=1pt, color=red] (1,0)-- (1,1.8);
\draw [line width=1pt ] (1,0)-- (3,0);
\draw [line width=1pt] (3,0)-- (5,0);

\draw [<-,line width=1pt,color=green] (-0.9,2.2)-- (0,4);
\draw [->,line width=1pt,color=red] (0,4)-- (0.9,2.2);

\begin{scriptsize}
\draw [fill=black] (-3,0) circle (2pt);
\draw[color=black] (-3.5,-0.5) node {\normalsize	$-1$};


\draw [fill=black] (1,0) circle (2pt);
\draw[color=black] (1,-0.5) node {\normalsize	$-1$};

\draw [fill=black] (3,0) circle (2pt);
\draw[color=black] (3,-0.5) node {\normalsize $-2$};
\draw[color=black] (3,-1.3) node {\normalsize$\mathcal{G}_2$};

\draw [fill=black] (5,0) circle (2pt);
\draw[color=black] (5,-0.5) node {\normalsize$-2$};

\draw [fill=black] (-3,-2) circle (2pt);
\draw[color=black] (-3.6,-2) node {\normalsize$-3$};
\draw[color=black] (-3,-2.8) node {\normalsize$\mathcal{G}_1$};

\draw [fill=black] (-3,2) circle (2pt);
\draw[color=black] (-3.6,2) node {\normalsize$-2$};


\draw [fill=black] (0,4) circle (2pt);
\draw[color=black] (-0.6,4) node {\normalsize$-1$};
\draw[color=black] (0.7,4) node {\normalsize$\mathcal{G}_3$};
\end{scriptsize}
\end{tikzpicture}
\caption{There are three singular points of $\widetilde{C}\cup L_\infty$, say $p_1,p_2$ and $p_3$ with dual resolution graphs $\mathcal{G}_1, \mathcal{G}_2$ and $\mathcal{G}_3$, respectively.}
    \label{fig:cut}
\end{figure}
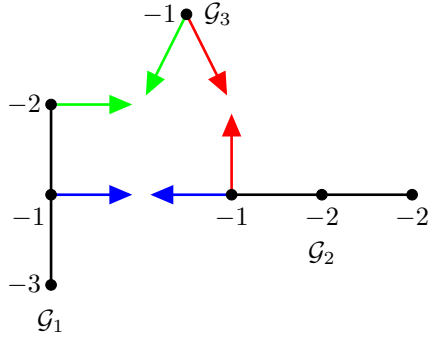

 The colors tell us the relation between branches and irreducible components. 
 This and the dual resolution graphs  $\mathcal{G}_1, \mathcal{G}_2$ and $\mathcal{G}_3$ (see Fig. \ref{fig:cut})  are sufficient to determine the topological type of the germ of the irreducible components at each of its singular points.
 To see this we subject the graphs $\mathcal{G}_1, \mathcal{G}_2$ and $\mathcal{G}_3$ repeatedly to a contraction operation which corresponds to blowing down of a curve.
 We call a vertex in a dual resolution graph (not associate to a minimal resolution) \textbf{contractible} when it has label $-1$ and valency less than three.
 Contraction of one of these vertices consists:
 \begin{itemize}
     \item if this vertex has valency 2, in adding one to the intersection numbers of its labeled adjacent vertices, removing the vertex, and  amalgamating its two adjacent edges into one edge.
     \item if this vertex has valency 1, in adding one to the intersection numbers of its labeled adjacent vertex and removing the vertex and its adjacent edge. 
 \end{itemize}
 
 \begin{definition}
  The \textbf{contraction process} of a dual resolution graph of a germ of a complex algebraic plane curve  $(\Gamma,p)$ with respect to one of its irreducible components $(\Gamma',p)$ consists in removing the arrow vertices and the edges connected to it except the ones which corresponds to the branches of $(\Gamma',p)$.
  In the resulting graph one repeatedly applies all possible contractions. The non-contractible graph finally obtained is the dual resolution graph of $(\Gamma',p)$.     
 \end{definition}
 
  Notice we get only arrow vertex at the end of contraction process if and only if $(\Gamma',p)$ is smooth. 
  For instance, to determine the  dual resolution graph of $(B,p_2)$ one removes the   red and arrow vertex and applies three contractions (see Fig. \ref{fig:contraction}).

 \begin{figure}[h]
     \centering
     
 \begin{tikzpicture}[line cap=round,line join=round,>=triangle 45,x=0.7cm,y=0.7cm]
\draw [line width=1pt] (0,0)-- (2,0);
\draw [line width=1pt] (2,0)-- (4,0);
\draw [->,line width=1pt,color=blue] (0,0) -- (0,2);
\draw [line width=1pt] (6,0)-- (10,0);
\draw [->,line width=1pt,color=blue] (6,0) -- (6,2);
\draw [->,line width=1pt,color=blue] (12,0) -- (12,2);
\draw [->,line width=1pt,color=blue] (13,1.7) -- (13,2.02);
\begin{scriptsize}
\draw [fill=black] (0,0) circle (2pt);
\draw[color=black] (0,-0.5) node {\normalsize$-1$};
\draw [fill=black] (2,0) circle (2pt);
\draw[color=black] (2,-0.5) node {\normalsize$-2$};
\draw [fill=black] (4,0) circle (2pt);
\draw[color=black] (4,-0.5) node {\normalsize$-2$};
\draw [fill=black] (6,0) circle (2pt);
\draw[color=black] (6,-0.5) node {\normalsize$-1$};
\draw [fill=black] (10,0) circle (2pt);
\draw[color=black] (10,-0.5) node {\normalsize $-2$};
\draw [fill=black] (12,0) circle (2pt);
\draw[color=black] (12,-0.5) node {\normalsize $-1$};
\end{scriptsize}
\end{tikzpicture}
\caption{Contraction process of the germ $(B,p_2)$.}
     \label{fig:contraction}
 \end{figure}
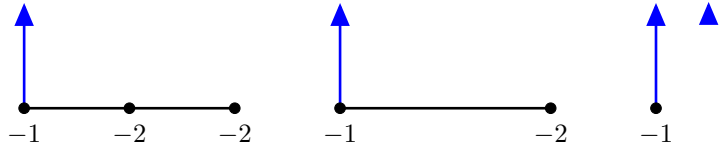

 To determine the  dual resolution graph of $(B,p_1)$ one removes the green  edge and its arrow vertex: there are no contractible vertices (see Fig. \ref{fig:bluegraph}).
\begin{figure}[h]
    \centering
 \begin{tikzpicture}[line cap=round,line join=round,>=triangle 45,x=0.8cm,y=0.8cm]
\draw [line width=1pt] (0,0)-- (2,0);
\draw [line width=1pt] (2,0)-- (4,0);
\draw [->,line width=1pt,color=blue] (2,0) -- (2,2);
\begin{scriptsize}
\draw [fill=black] (0,0) circle (2pt);
\draw[color=black] (-0.14,-0.5) node {\normalsize$-3$};
\draw [fill=black] (2,0) circle (2pt);
\draw[color=black] (2,-0.5) node {\normalsize$-1$};
\draw [fill=black] (4,0) circle (2pt);
\draw[color=black] (4,-0.5) node {\normalsize$-2$};
\end{scriptsize}
\end{tikzpicture}
\caption{Minimal resolution graph of the germ $(B,p_1)$.}
    \label{fig:bluegraph}
\end{figure}
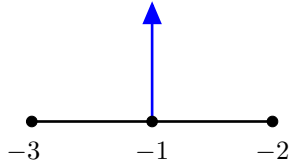
Thus, we extract from the Lipschitz graph the following data:

\begin{itemize}
    \item the number of irreducible components.
    
    \item there are three singular points, say $p_1,p_2$ and $p_3$ with dual resolution graphs $\mathcal{G}_1, \mathcal{G}_2$ and $\mathcal{G}_3$, respectively.
    By \cite[Theorem 8.1.7]{Wall}, this is equivalent to know the topological type of the germs $(\widetilde{C}\cup L_\infty,p_1), (\widetilde{C}\cup L_\infty,p_2)$ and $(\widetilde{C}\cup L_\infty,p_3)$. 
    
    \item the dual resolution graphs of the germs $(G,p_1),(B,p_1),(B,p_2)$ and $(G,p_3)$, obtained by the contraction process. 
    By \cite[Theorem 8.1.7]{Wall}, this is equivalent to know the topological type of these germs.
\end{itemize}
\end{example}
From the above example it is easy to see that the equivalence $(\ref{itii})\Leftrightarrow(\ref{itiii})$ of Theorem \ref{global} holds. Now, we deal with the equivalence between (\ref{iti}) and (\ref{itii}).

\begin{proof}[Proof of $(\ref{iti}) \Leftrightarrow (\ref{itii})$ of Theorem \ref{global}]

We start assuming that there exists a bilipschitz map $\phi: C\to \Gamma$.  
By \cite[Theorem 1.1]{Anne}, for each singular point  $p\in C$ the topological type of the germ $(C,p)$ is the same as the topological type of $(\Gamma,\phi(p))$. 
By item (\ref{it2}) of Theorem \ref{atinfinity}, there is a bijection $\psi$ between the set of points at infinity of $C$ and the set of points at infinity of $\Gamma$ such that $(\widetilde{C}\cup L_\infty,p)$ has the same topological type as $(\widetilde{\Gamma}\cup L_\infty,\psi(p))$. 

Restricting $\phi$ to smooth points of $C$ we get a homeomorphism between $C\backslash \Sigma(C)=\bigcup_{i\in I}C_i\backslash \Sigma(C)$ and $\Gamma\backslash \Sigma(\Gamma)=\bigcup_{j\in J}\Gamma_j\backslash \Sigma(\Gamma)$, where $\Sigma(C)$ and $\Sigma(\Gamma)$ denote the singular points of $C$ and $\Gamma$, respectively.
Since $C_i, \Gamma_j$ are irreducible and $\Sigma(C)$ and $\Sigma(\Gamma)$ are finite, $C_i\backslash\Sigma(C)$ and $\Gamma_j\backslash\Sigma(\Gamma)$ are connected.
Then the map  $\sigma:I\to J$, defined by 
$\sigma(i)=j\in J$ if and only if  $\phi(C_i\backslash\Sigma(C))= \Gamma_j\backslash\Sigma(\Gamma)$,  is a bijection.

We extend the application $\phi|_{C_i\backslash \Sigma(C)}$ to topological closure of $C_i\backslash \Sigma(C)$ and we get the bilipschitz map $\phi_i:C_i\to\Gamma_{\sigma(i)}, \phi_i=\phi|_{C_i}$. 
Applying  \cite[Theorem 1.1]{Anne} to   $\phi_i:C_i\to\Gamma_{\sigma(i)}$, we  obtain that for each singular point  $p\in C_i$ the topological type of the germ $(C_i,p)$ is the same as the topological type of $(\Gamma_{\sigma(i)},\phi(p))$. 

By item (\ref{it2}) of Theorem \ref{atinfinity}, there is a bijection $\psi_i$ between the set of points at infinity of $C_i$ and the set of points at infinity of $\Gamma_{\sigma(i)}$ such that $(\widetilde{C}_i\cup L_\infty,p)$ has the same topological type as $(\widetilde{\Gamma}_{\sigma(i)}\cup L_\infty,\psi_i(p))$. 
Moreover, $\psi_i$ can be chosen to be the restriction of $\psi$ to the points at infinity of $C_i$. 

Recall the parametrization $\iota:\C^2\to\P^2$ of $\P^2$ given by $\iota(x,y)=[x:y:1]$. 
Then the bijection $\varphi:\Sigma(\widetilde{C}\cup L_\infty)\to\Sigma(\widetilde{\Gamma}\cup L_\infty)$ defined by 

$$\varphi(p)=\begin{cases}
\psi(p)& \text{if $p\in L_\infty$},\\
\iota(\phi(\iota^{-1}(p)))& \text{otherwise},
\end{cases}$$
give us the bijection of item (\ref{itii}) of Theorem \ref{global}.

Now, the reciprocal, i.e., that (\ref{itii}) implies (\ref{iti}) of Theorem \ref{global}. 
We can assume that $I=J=\{1,\ldots,m\}$ and $\sigma=\mathrm{id}$. The item (\ref{itii}) of Theorem \ref{global} implies that both curves $C_i$ and $\Gamma_i$ have the same number of points at infinity and singular points for $i=0,\ldots,m$ where $C_0=C$ and $ \Gamma_0=\Gamma$.

Let $p_1,\ldots,p_{s}$ be the singular points of $\widetilde{C}$  and let $q_1,\ldots,q_{s}$ be the singular points of $\widetilde{\Gamma}$ which are not point at infinity of $C$ and $\Gamma$, respectively.
We denote in such a way that $(\widetilde{C}_i,p_l)$ has the same topological type as $(\widetilde{\Gamma}_i,q_l)$ for $l=1,\ldots,s$ and $i=0,1,\ldots, m$.

Similarly,  let $p_{s+1},\ldots,p_{m}$ be points at infinity of $C$ and let $q_{s+1},\ldots,q_{m}$ be the points at infinity of $\Gamma$  denoted in such a way that $(\widetilde{C}_i\cup L_\infty,p_l)$ has the same topological type as $(\widetilde{\Gamma}_i\cup L_\infty,q_l)$ for $l=s+1,\ldots,m$ and $i=0,1,\ldots,m.$

Let $B(p_l)\subset \P^2$ be a regular coordinate ball, that is, there exists a smooth coordinate ball $B'(p_l)\supseteq
\overline{B(p_l)}$. 
Shrinking
$B'(p_l)$ if necessary, we can assume that  $B'(p_l)\cap B'(p_j)=\emptyset$ for $l\neq j$ and we can apply
 \cite[Theorem 1.1]{Anne}, i.e., for $l=1,\ldots,s$ there exists a bilipschitz map  
 $$\phi_l:C\cap \iota^{-1}(B'(p_l))\to  \phi_l(C\cap \iota^{-1}(B'(p_l)))\subset\Gamma$$ which is biholomorphic except at $\iota^{-1}(p_l)$ and $\phi_l(\iota^{-1}(p_l))=\iota^{-1}(q_l)$.
 
 Similarly, by Theorem \ref{atinfinity}, there exists a bilipschitz $$\Phi: C\cap \Bigl(\bigcup_{l=s+1}^m\iota^{-1}(B'(p_l)) \Bigr)\to  \Phi\Bigl(C\cap\Bigl(\bigcup_{l=s+1}^m\iota^{-1}(B'(p_l))\Bigl)\Bigl)\subset \Gamma$$   
which is biholomorphic. Then the curve $C$ is almost covered with domains of bilipschitz maps. 
The part that is missing is inside of $C\backslash \Bigl(\bigcup_{l=1}^m\iota^{-1}(B(p_l)\Bigr)$ which is a union of connected compact orientable surfaces $K_i$ with boundary. More precisely, let $K_i=C_i\backslash \Bigl(\bigcup_{l=1}^m\iota^{-1}(B(p_l))\Bigr)$. 
Recall that connected compact orientable surfaces with boundary are classified up to diffeomorphism by the Euler characteristic and the number of connected components of boundary see, for instance, \cite[Chapter 9, Theorem 3.11]{hirsch1976differential}.

Let us calculate the Euler characteristic of $K_i$. Shrinking
$B'(p_l)$ if necessary, we assume that $B'(p_l)$ intersect $\widetilde{C}_i$ if and only if $p_l$ is a singular point of $\widetilde{C}_i\cup L_\infty$. By additive property of the Euler characteristic   we have:
\begin{multline*}
  \chi(\widetilde{C}_i)=\chi\Bigl(\Bigl(\widetilde{C}_i \backslash \bigcup_l  B(p_l)\Bigr)\cup \Bigl(\bigcup_l \overline{B(p_l)}\cap \widetilde{C}_i\Bigr)\Bigr) \\ =\chi\Bigl(\widetilde{C}_i \backslash \bigcup_l  B(p_l)\Bigr)+\chi\Bigl(\bigcup_l\overline{B(p_l)}\cap \widetilde{C}_i\Bigr)-\chi\Bigl(\Bigl(\widetilde{C}_i \backslash \bigcup_l  B(p_l)\Bigr)\cap \Bigl(\bigcup_l \overline{B(p_l)}\cap \widetilde{C}_i)\Bigr), 
\end{multline*}
  Note that  all spaces appearing  in the above equation are compact and triangulable. By the Conical Structure Theorem and the additive property we know that 
$$\chi\Bigl(\bigcup_l\overline{B(p_l)}\cap \widetilde{C}_i\Bigl)=\sum_l\chi\bigl(\overline{B(p_l)}\cap \widetilde{C}_i\bigl)=m_i,$$
where $m_i$ is the number of singular points of $\widetilde{C}_i\cup L_\infty$.
Shrinking each $B(p_l)$ if necessary, by \cite[Corollary 2.9]{milnor} we may assume that the boundary of $B(p_l)$ intersect $\widetilde{C}_i$ transversally. 
Thus there are two possibilities: $B(p_l)\cap \widetilde{C}_i=\emptyset$ or 
$B(p_l)\cap \widetilde{C}_i$ is a smooth compact 1-manifold.   
In the latter case,
by the classification theorem of smooth 1-manifold, the intersection is diffeomorphic to $\S^1$. In both cases we have
$$\chi\Bigl(\Bigl(\widetilde{C}_i \backslash \bigcup_l  B(p_l)\Bigr)\cap \Bigl(\bigcup_l \overline{B(p_l)}\cap \widetilde{C}_i\Bigr)\Bigr)=0.$$

On the other hand, \cite[Theorem 7.1.1]{Wall} tell us a formula for the Euler characteristic of a curve in terms of its degree and its singularities. More precisely,

$$\chi(\widetilde{C}_i)=3d_i-d_i^2 +\sum_{p\in \widetilde{C}_i}\mu_p(\widetilde{C}_i)$$
where $d_i$ denotes the degree of $\widetilde{C}_i
$ and $\mu_p(\widetilde{C}_i)$ denotes the Milnor number of $\widetilde{C}_i$ at $p$. 

Recall that the Milnor number is an invariant of the topological type, for instance, see \cite[Theorem 6.5.9]{Wall}. The degree of the curve in its turn is an invariant of the topological type of the germs  $(\widetilde{C}_i\cup L_\infty,p_l)$ since
\begin{equation*}
\deg \widetilde{C}_i=\sum_{p\in \widetilde{C}_i\cap L_\infty} (\widetilde{C}_i \cdot L_\infty)_p,
\end{equation*}
where $(\widetilde{C}_i\cdot L_\infty)_p$ denotes the intersection number between $\widetilde{C}_i$ and $L_\infty$. 
Also, the invariance of degree for the Lipschitz geometry at infinity of complex algebraic curves is given in \cite[Corollary 3.2]{invariancedegree}. Having said that, we have $\chi(\widetilde{C}_i)=\chi(\widetilde{\Gamma_i})$ and for $K_i$,
\begin{equation}\label{euler}    
\chi(K_i)=\chi(\widetilde{C}_i \backslash \cup_l  B(p_l))=\chi(\widetilde{C}_i)-m_i.
\end{equation}

Let $\mathcal{B}(q_l)=\iota\Bigl(\Phi\bigl(C\cap\iota^{-1}(B(p_l))\bigl)\Bigl)\cup\{p_l\} $ for $l=s+1,\ldots,m$
and $\mathcal{B}(q_l)=\iota\Bigl(\phi_l\bigl(C\cap \iota^{-1}(B(p_l))\bigr)\Bigr)$ for $l=1,\ldots,s$ and   $$\mathcal{K}_i=\Gamma_i\backslash \bigcup_l\phi_l(C\cap \iota^{-1}(B(p_l)))\cup  \Phi(C\cap(\bigcup_{l=s+1}^m\iota^{-1}(B(p_l))).$$
To calculate the Euler characteristic of $\mathcal{K}_i$ we notice that 
$\mathcal{K}_i=\Gamma_i\backslash \Bigl(\bigcup_{l=1}^m\iota^{-1}(\mathcal{B}(q_l))\Bigr).$ And by similar arguments as above one has 
\begin{equation}\label{eulerG}    
\chi(\mathcal{K}_i)=\chi\Bigl(\widetilde{\Gamma}_i \backslash \bigcup_l  \mathcal{B}(q_l)\Bigr)=\chi(\widetilde{\Gamma}_i)-m_i.
\end{equation}

It follows from equation (\ref{euler}) and (\ref{eulerG}) that $K_i$ and $\mathcal{K}_i$ have   the same Euler characteristic.
The map $f_i:\partial K_i\to \partial \mathcal{K}_i$ defined by the restrictions $$\phi_l|_{C_i\cap\iota^{-1}(\partial B(p_l))}\text{ for } l=1,\ldots,s \text{ and } \Phi|_{C_i\cap(\iota^{-1}(\partial B(p_l)))} \text{ for } i=s+1,\ldots, m$$ is bihomorphic.
Now, we use a  slight generalization of the classification of smooth compact surface with boundary:

\begin{lemma}\label{classification}
   Let $K_i$ and $\mathcal{K}_i$ be connected compact orientable smooth surfaces with boundary and let $f_i:\partial K_i\to \partial \mathcal{K}_i$ be an orientation-preserving diffeomorphism. Then $f_i$ extends to a diffeomorphism $F_i:K_i\to \mathcal{K}_i$ if only if $K_i$ and $\mathcal{K}_i$ have the same Euler characteristic.
\end{lemma}
\begin{proof}
The boundaries $\partial K_i, \partial \mathcal{K}_i$ are smooth compact 1-manifolds and thus its connected components are diffeomorphic to $\S^1$. Since $f_i$ is a diffeomorphism between $\partial K_i$ and $ \partial \mathcal{K}_i$, they have the same number of connected components. 
Let $g_i:K_i\to \mathcal{K}_i$ be the diffeomorphism given by the classification of smooth compact surface theorem \cite[Chapter 9, Theorem 3.11]{hirsch1976differential}. 
Up to isotopy every orientation-preserving diffeomorphism of $\S^1$ is the identity \cite[Chapter 8, Theorem 3.3]{hirsch1976differential}, then we know that the restriction map   $g_i|_{\partial K_i}$, and $f_i$  are isotopic, say by $H_i:[0,1]\times \partial K_i\to \partial \mathcal{K}_i, H_i(0,\cdot)=f_i, H_i(1,\cdot)=g_i|_{\partial K_i}$.  

The collar neighborhood theorem \cite[Theorem 9.26]{leesmooth} shows that $\partial K_i$ has a collar neighborhood $\mathcal{C}$ in $K_i$; which is
the image of a smooth embedding $E: [0,1)\times \partial K_i\to K_i$ satisfying $E(0,x)=x$
for all $x\in\partial K_i$. To simplify notation, we use this embedding to identify $\mathcal{C}$
with $[0,1)\times \partial K_i$ and denote a point in $\mathcal{C}$ as an ordered pair $(s, x)$  with $s\in[0,1)$ and $x \in\partial K_i$; thus $(s, x)\in \partial K_i$ if and only if $s = 0$. For any $a\in(0,1)$, let $\mathcal{C}(a)=\{(s,x)\in \mathcal{C}: 0\leq s<a\}$ and $K_i(a)=K_i\backslash \mathcal{C}(a)$.

Let $\gamma: [0,1]\to [0,1]$ be a smooth map that satisfies $\gamma(0)=0$ and $\gamma(s)=1$
for $\frac{1}{2}\leq s\leq 1$. Define $F:K_i\to \mathcal{K}_i$ by
$$F_i(p)=\begin{cases}
g_i(p),& \text{if $p\in \mathrm{Int} K_i\left(\frac{1}{2}\right)$,}\\
(s,H_i(x,\gamma(s)),& p=(s,x)\in \mathcal{C}.
\end{cases}$$

These definitions both give the map $g_i$ on the set $\mathcal{C}\backslash \overline{\mathcal{C}\left(\frac{1}{2}\right)}$
where they overlap,
so $F_i$ is a diffeomorphism extension of $f_i$.
\end{proof}

The map $F:K_i\to \mathcal{K}_i$ is a diffeomorphism between compact sets, so it is a bilipschtz map. 
The maps $F, \Phi, \phi_l$ agree on the components of the boundary of $K_i$.
It follows  that there exists a bilipschitz  map $\Psi:C\to \Gamma$ such that $\Psi|_{K_i}=F_i,\: \Psi|_{\iota^{-1}(B(p_l))\cap C}=\phi_l$ and    
$\Psi|_{C\cap\bigl(\bigcup_{l=s+1}^m\iota^{-1}(B(p_l))\bigl)}=\Phi$. 
\end{proof}


\end{document}